\documentclass[11pt,twoside,a4paper]{amsart}

\usepackage{ stmaryrd }

\usepackage[
  backend=biber,
  style=alphabetic,
  giveninits=true,
  maxbibnames=5,  
  minbibnames=3,  
  maxcitenames=2, 
  url=false,
  doi=true,
  isbn=false,
  eprint=true,
]{biblatex}

\usepackage{amsmath,amsfonts,amssymb,amsthm}
\usepackage[all,cmtip]{xy}
\usepackage{tikz-cd}

\usepackage[colorlinks]{hyperref}
\hypersetup{
    colorlinks,
    citecolor=magenta,
    filecolor=red,
    linkcolor=blue,
    urlcolor=purple,
    pdftitle = L-equivalence for cubic fourfolds
}
\usepackage{mathtools}

\usepackage{cleveref}
\usepackage[T2A]{fontenc}
\DeclareSymbolFont{cyrillic}{T2A}{cmr}{m}{n}
\DeclareMathSymbol{\Sha}{\mathalpha}{cyrillic}{216}
\usepackage{xcolor}
 
\usepackage[normalem]{ulem}

\numberwithin{equation}{section}
\usepackage{microtype}
\newtheorem{dummy}{dummy}[section]

\newtheorem{theorem}[dummy]{Theorem}
\newtheorem{corollary}[dummy]{Corollary}
\newtheorem{lemma}[dummy]{Lemma}
\newtheorem{proposition}[dummy]{Proposition}

\theoremstyle{definition}
\newtheorem{definition}[dummy]{Definition}
\newtheorem{remark}[dummy]{Remark}
\newtheorem{example}[dummy]{Example}
\newtheorem{question}[dummy]{Question}
\theoremstyle{plain}

\newtheorem{conjecture}[dummy]{Conjecture}

\newcommand{\Z}{\ensuremath{\mathbb{Z}}}

\newcommand{\Q}{\ensuremath{\mathbb{Q}}}
\newcommand{\CC}{\ensuremath{\mathbb{C}}}
\newcommand{\PP}{\ensuremath{\mathbb{P}}}
\newcommand{\LL}{\ensuremath{\mathbb{L}}}
\newcommand{\FM}{\ensuremath{\operatorname{FM}}}
\newcommand{\sgn}{\ensuremath{\operatorname{sgn}}}

\newcommand{\bigo}{\ensuremath{\mathcal{O}}}
\newcommand{\klat}{\ensuremath{\Lambda_{\operatorname{K3}}}}

\newcommand{\crs}{\ensuremath{\operatorname{crs}}}
\newcommand{\cublat}{\ensuremath{\Lambda_{\text{cub}}^0}}
\newcommand{\fullcublat}{\ensuremath{\Lambda_{\text{cub}}}}

\def\Db{\calD^{b}}

\newcommand{\gravk}{\ensuremath{K_0(\operatorname{Var}_k)}}
\newcommand{\gravc}{\ensuremath{K_0(\operatorname{Var}_\CC)}}
\newcommand{\gghs}{\ensuremath{K_0^{\oplus}(\operatorname{HS}_{\Z,2})}}
\newcommand{\gghsf}{\ensuremath{K_0^{\oplus}(\operatorname{HS}_{\Z})}}
\newcommand{\gghsn}{\ensuremath{K_0^{\oplus}(\operatorname{HS}_{\Z,n})}}

\def\disc{\operatorname{disc}}

\def\Kn{{\ensuremath{\operatorname{K3}^{[n]}}}}

\def\dv{\operatorname{div}}

\DeclareMathOperator{\Aut}{Aut}

\DeclareMathOperator{\rk}{rk}

\DeclareMathOperator{\GL}{GL}

\DeclareMathOperator{\Hom}{Hom}
  
\DeclareMathOperator{\End}{End}

\DeclareMathOperator{\id}{id}

\DeclareMathOperator{\Emb}{Emb}

\DeclareMathOperator{\Spec}{Spec}

\DeclareMathOperator{\SHom}{\mathcal{H\kern -1pt}\textit{om}} 
\DeclareMathOperator{\SEnd}{\mathcal{E\kern -1pt}\textit{nd}} 
\DeclareMathOperator{\SExt}{\mathcal{E\kern -1pt}\textit{xt}} 

\DeclareMathOperator{\NS}{NS}

\def\dar[#1]{\ar@<2pt>[#1]\ar@<-2pt>[#1]}

\newcommand\bbC{\mathbb{C}}

\newcommand\bbQ{\mathbb{Q}}

\newcommand\bbZ{\mathbb{Z}}

\newcommand\calA{\mathcal{A}}

\newcommand\calC{\mathcal{C}}
\newcommand\calD{\mathcal{D}}

\newcommand\calX{\mathcal{X}}
\newcommand\calY{\mathcal{Y}}

\newcommand{\function}[5]{\begin{array}{lrcl} 
#1: & #2 & \longrightarrow & #3 \\
    & #4 & \longmapsto & #5 \end{array}}

\newcommand{\isomorphism}[5]{\begin{array}{lrcl} 
#1: & #2 & \overset{\sim}{\longrightarrow} & #3 \\
    & #4 & \longmapsto & #5 \end{array}}

\newcommand{\functionstar}[4]{
\begin{array}{rcl} #1 &\longrightarrow &#2 \\ #3&\longmapsto &#4 \end{array}
}

\newcommand{\prim}{\operatorname{prim}}

\title{L-equivalence and Fourier--Mukai partners of cubic fourfolds}

\author[R.~Meinsma]{Reinder Meinsma}
\address{
Fakultät für Mathematik und Informatik,
Universität des Saarlandes,
Campus E2.4, 66123 Saarbrücken, Germany
}
\email{meinsma@math.uni-sb.de}

\author[R. Moschetti]{Riccardo Moschetti}
\address{Department of Mathematics F. Casorati, University of Pavia, via Ferrata 5, 27100 Pavia, Italy} 
\email{riccardo.moschetti@unipv.it}
\begin{document}
\begin{abstract}
We study $\LL$-equivalence in the Grothendieck ring of varieties and its interaction with categorical invariants of cubic fourfolds. Assuming a Derived Torelli-type criterion for Kuznetsov components and a mild condition on the discriminant of the transcendental lattice, we prove a counting formula for Fourier--Mukai partners of cubic fourfolds. 
Using this formula together with direct primitive-extension computations, we exhibit cubic fourfolds with a prescribed algebraic lattice admitting a unique non-isomorphic Fourier--Mukai partner with the same primitive algebraic lattice; the two cubic fourfolds are trivially $\LL$-equivalent. Finally, we show that $\LL$-equivalence classes of cubic fourfolds are finite.
\end{abstract}
\maketitle
\setcounter{tocdepth}{1}
\tableofcontents
\section{Introduction}
The Grothendieck ring of varieties over a field $k$ is the quotient of the free abelian group generated by isomorphism classes of varieties over $k$, by the subgroup generated by expressions of the form 
\[
[X] - [Z] = [X\setminus Z],
\]
where $Z\subset X$ is a closed subvariety. Multiplication in $\gravk$ is defined by the fibre product over $k$:
\[
[X]\cdot [Y] = [X\times_kY].
\]
The class $\LL\coloneqq [\mathbb{A}^1]\in \gravk$ is called the \textit{Lefschetz motive}. The Lefschetz motive is a zero-divisor in $\gravc$, as was proved by Borisov \cite{Bor17}. Following this observation, Kuznetsov and Shinder defined $\LL$-equivalence:
\begin{definition}[\cite{KS18}]
    Two varieties $X,Y$ over a field $k$ are \textit{$\LL$-equivalent}, denoted $X\sim_\LL Y$, if there exists a $n\geq 0$ such that 
    \[
    \LL^n([X]-[Y]) = 0 \in \gravk.
    \]
\end{definition}

$\LL$-equivalence seems to be closely related to derived equivalence, especially for projective hyperk\"ahler manifolds. At the time of writing, many examples of $\LL$-equivalent varieties have been constructed. In the case of K3 surfaces, see \cite{KS18, HL18, IMOU20, KKM20} and \cite{SZ20}, also related with $\LL$-equivalent degree five elliptic curves. In higher dimensions, $\LL$-equivalent Calabi-Yau varieties are considered in \cite{Bor17, Martin2016, OR2018, IMOU2019, KR2019, Man2019, BCP2020}. Examples of $\LL$-equivalence for hyperk\"ahler manifolds appear in \cite{Oka21, BB2025}, for Gushel-Mukai surfaces in \cite{PMS2025}, for cubic fourfolds in \cite{FL23}.

All currently known examples of $\LL$-equivalent hyperk\"ahler manifolds are also derived equivalent. In fact, for very general K3 surfaces, $\LL$-equivalence implies derived equivalence \cite{Mei24}. On the other hand, there exist examples of derived equivalent hyperk\"ahler manifolds of \Kn-type that are not $\LL$-equivalent, for any $n\geq 2$ \cite{Mei25}. Notably, there are no known examples of K3 surfaces that are derived equivalent but not $\LL$-equivalent.
\begin{conjecture} \cite{Mei25} \label{conj: L implies D for hyperkahler}
    If $X$ and $Y$ are $\LL$-equivalent projective hyperk\"ahler manifolds, then they are derived equivalent.
\end{conjecture}

In this paper, we study $\LL$-equivalence and Fourier--Mukai partnership for cubic fourfolds over the complex numbers. 
For a cubic fourfold $X$, there exists a canonical semiorthogonal decomposition 
\[
\Db(X) \simeq \langle \mathcal{A}_X, \bigo_X, \bigo_X(1), \bigo_X(2) \rangle.
\]
The component $\mathcal{A}_X$ is called the \textit{Kuznetsov component} of $\Db(X)$. Since cubic fourfolds are Fano varieties (in the sense that their anti-canonical bundles are ample), 
it follows from Bondal and Orlov's Reconstruction Theorem \cite{BO01} that any Fourier--Mukai equivalence $\Db(X)\simeq \Db(Y)$ between the derived categories of cubic fourfolds induces an isomorphism $X\simeq Y$. However, there are many examples of non-isomorphic cubic fourfolds $X$, $Y$ whose derived categories have equivalent Kuznetsov components $\mathcal{A}_X\simeq \mathcal{A}_Y$. Such cubic fourfolds are called \textit{Fourier--Mukai partners}.

We therefore propose the following conjecture related to Conjecture \ref{conj: L implies D for hyperkahler}. This conjecture has also recently appeared in \cite{BB2025}.
\begin{conjecture} \label{conj:LimpliesFM}
    If $X$ and $Y$ are $\LL$-equivalent cubic fourfolds, then they are Fourier--Mukai partners:
    \[
    X\sim_\LL Y \implies \mathcal{A}_X\simeq \mathcal{A}_Y
    \]
\end{conjecture}

\begin{remark}
Conjecture \ref{conj:LimpliesFM} should also follow from Kontsevich’s expectation that semiorthogonal decompositions are canonical (up to mutation) and functorial with respect to geometric operations. See, for further references \cite{ESS2025, HLNotes2024}.
\end{remark}

We note that the $\LL$-equivalence relation in the context of cubic fourfolds was first studied by Galkin and Shinder \cite{GS14}. In that paper, the so-called \textit{rational defect} $\mathcal{M}_X$ of a cubic fourfold $X$ plays a crucial role. The rational defect is defined to be 
\[
\mathcal{M}_X \coloneqq \frac{[X] - [\PP^4]}{\LL}\in \gravc[\LL^{-1}].
\]
If $X$ is rational, the equation $[X] = [\PP^4] + \mathcal{M}_X\cdot\LL$ holds in the (non-localised) Grothendieck ring \gravc, and $\mathcal{M}_X$ is a sum of smooth projective surfaces \cite[Lemma 2.1]{GS14}. ationality of special cubic fourfolds has been established for several Hassett divisors by explicit geometric constructions, see for instance \cite{RS19}.
In fact, in all of the cases where we can compute $\mathcal{M}_X$ for rational cubic fourfolds, we have the equation \cite{KS18,IMOU20}
\begin{equation}\label{eq: motive of certain rational cubic fourfolds}
    [X] = 1 + \LL^2 + \LL^4 + [S]\LL,
\end{equation}
where $S$ is a K3 surface. 

In this paper, we construct pairs of $\LL$-equivalent cubic fourfolds $X$, $Y$ whose classes in the Grothendieck ring are of the form \eqref{eq: motive of certain rational cubic fourfolds}: 

\begin{proposition}[See Proposition \ref{prop: L10 examples} and Proposition \ref{prop: L29 examples}]
\label{prop: L10 examples intro}
Suppose that $X$ satisfies Derived Torelli and is very general in one of the following loci:
\begin{enumerate}
\item[(i)] The algebraic lattice $A(X)\subset H^4(X,\Z)$ is isometric to
$$L_{10}\coloneqq
\begin{pmatrix}
3 & 1 & 1\\
1 & 3 & 0\\
1 & 0 & 10
\end{pmatrix}.
$$
\item[(ii)] The marked algebraic lattice $(A(X),h^2)$ is isometric to
$$L_{2,7}\coloneqq
\begin{pmatrix}
3 & 4 & 0\\
4 & 10 & 2\\
0 & 2 & 14
\end{pmatrix},$$
with the first basis vector corresponding to $h^2$.
\end{enumerate}
Then $X$ has a non-isomorphic Fourier--Mukai partner $Y$ for which there exists a K3 surface $S$ such that
$$
[X]=[Y]=1+\LL^2+\LL^4+[S]\LL\in\gravc.
$$

In case {\rm (i)}, $Y$ is the unique non-trivial Fourier--Mukai partner of $X$.

In case {\rm (ii)}, the cubic fourfold $Y$ is the unique non-isomorphic Fourier--Mukai partner satisfying $A(Y)_{\mathrm{prim}}\simeq A(X)_{\mathrm{prim}}$.
Moreover, the marked algebraic lattice $(A(Y),h^2)$ is isometric to
$L_{2,7}$, and $X,Y\in\mathcal C_{14}\setminus\mathcal C_8$.
\end{proposition}

The $\LL$-equivalent cubic fourfolds in Proposition \ref{prop: L10 examples intro} are \textit{trivially} $\LL$-equivalent, in the sense that the motives $[X]$ and $[Y]$ are \textit{equal} in $\gravc$. Other examples of trivially $\LL$-equivalent cubic fourfolds can be found in \cite{FL23}.
\begin{question}
    Do there exist non-trivially $\LL$-equivalent cubic fourfolds?
\end{question}

The example associated with $L_{10}$ follows from a counting formula for Fourier--Mukai partners. For the family of cubic fourfolds in $\mathcal C_{14}$ containing the example $L_{2,7}$, the discriminant of the primitive algebraic lattice is divisible by $3$, and the counting formula does not apply. We treat this case instead by a direct computation of primitive extensions in the relevant transcendental component.
We will work with cubic fourfolds satisfying the following definition.
\begin{definition} \label{def: derived torelli assumption}
    A cubic fourfold $X$ is said to \textit{satisfy Derived Torelli} if it has the property that for any cubic fourfold $Y$, the following two statements are equivalent: 
    \begin{enumerate}
        \item
        $X$ and $Y$ are Fourier--Mukai partners, that is, we have an equivalence $\mathcal{A}_X\simeq \mathcal{A}_Y$.
        \item There is a Hodge isometry between transcendental lattices $T(X)\simeq T(Y)$.
    \end{enumerate}
\end{definition}
Note that $(1) \implies (2)$ always holds, see Remark \ref{rem:1 implies 2 in the assumption}.
There is a large class of cubic fourfolds which satisfy Derived Torelli, for instance if $X$ is a general cubic fourfold not contained in a Hassett divisor (see also Proposition \ref{prop: derived torelli criteria} and Remark \ref{cases in which the assumption doesn't hold}).

The algorithm of \cite{BBM25} counts virtual Fourier--Mukai partners by using suitable primitive extensions and then determines which of them are realised by cubic fourfolds; it is developed under the assumption
$O_{\mathrm{Hodge}}(T(X))={\pm\id}$.
Under the additional hypothesis that $\disc(T(X))$ is not divisible by $3$, the next result gives, in this setting, a closed formula for the actual Fourier--Mukai partners, allowing an arbitrary group $O_{\mathrm{Hodge}}(T(X))$. When $O_{\mathrm{Hodge}}(T(X))={\pm\id}$, our formula recovers the corresponding orbit count of \cite{BBM25}. This is also analogous to the counting formulas for Fourier--Mukai partners of K3 surfaces in \cite{HLOY02,Ma10}.

\begin{theorem}[Theorem \ref{thm: counting formula}] \label{thm: counting formula INTRO}
    Let $X$ be a cubic fourfold which satisfies Derived Torelli. Assume moreover that $\disc(T(X))$ is not divisible by $3$. Then we have 
    \begin{equation*}
    \#\FM(X) = \sum_{N'\in \mathcal{H}(N)}|O(N')\backslash O(A_{T(X)})/O_{\mathrm{Hodge}}(T(X))|.
    \end{equation*}
    Here, $N = A(X)_{\text{prim}}$ and the action of $O(N')$ on $O(A_{T(X)})$ is the one defined by Remark \ref{rem: isometry actions}, and $\mathcal{H}(N)$ is the set of lattices in the genus of $N$ which contain no vectors of square $2$, nor any vectors of square $6$ and divisibility $3$ or $6$. 
\end{theorem}
We also prove a partial counting formula for Fourier--Mukai partners of a cubic fourfold $X$ of which the primitive algebraic lattice $N\coloneqq A(X)_\mathrm{prim}$ has discriminant indivisible by $3$. For such a cubic fourfold, we prove in Proposition \ref{prop: counting formula div 3} that the number of Fourier--Mukai partners $Y$ which satisfy $A(Y)_\mathrm{prim} \simeq N'\in \mathcal{H}(N)$ is given by
\[
|O(N')\backslash O(A_{N})/O_{\mathrm{Hodge}}(T(X))|.
\]
The family considered in Proposition \ref{prop: L29 examples} lies outside the scope of this formula: for
$$N_{a,b}=
\begin{pmatrix}
42 & 3a\\
3a & 2b
\end{pmatrix}$$
one has $\disc(N_{a,b})=3(28b-3a^2)$.
We therefore study this family by computing primitive extensions separately in each transcendental component. For $L_{2,7}$, the component containing the prescribed marked algebraic lattice gives
exactly two Fourier--Mukai classes, both having full marked algebraic lattice isometric to $L_{2,7}$.

Finally, we study the following equivalence relations for cubic fourfolds (we copy the notation from \cite{BFM24}). For two cubic fourfolds $X$, $Y$, we write:
\begin{enumerate}
    \item[\textbf{BE}:] $X$ and $Y$ are birational: $X\overset{\simeq}{\dashrightarrow} Y$
    \item[\textbf{BF}:] $X$ and $Y$ have birational Fano varieties of lines: $F(X)\overset{\simeq}{\dashrightarrow} F(Y)$
    \item[\textbf{FM}:] The Kuznetsov components of $X$ and $Y$ are equivalent: $\mathcal{A}_X \simeq \mathcal{A}_Y$
    \item[\textbf{DF}:] The Fano varieties of lines of $X$ and $Y$ are derived equivalent: $\Db(F(X)) \simeq \Db(F(Y))$
    \item[\textbf{LE}:] $X$ and $Y$ are $\LL$-equivalent: $X\sim_\LL Y$
    \item[\textbf{LF}:] The Fano varieties of lines of $X$ and $Y$ are $\LL$-equivalent: $F(X) \sim_\LL F(Y)$
\end{enumerate}

\begin{theorem}[Theorem \ref{thm: L equivalences classes are finite}, Theorem \ref{thm: L implies FM}, Proposition \ref{prop: LE implies LF}, Corollary \ref{cor: BF implies LF}, Theorem \ref{thm: LF implies FM and DF}]
\label{thm: main theorem intro}
    For two cubic fourfolds $X$, $Y$, we have the following implications, where the implications marked by $\dagger$ are proved only when $X$ and $Y$ satisfies Derived Torelli and $X$ is such that $\End_\mathrm{Hodge}(T(X)) \simeq \Z$:
    \begin{equation} \label{eq: our implications diagram}
    \xymatrix{
    & \mathbf{LE} \ar@{=>}[dl]_{\dagger} & \mathbf{LF} \ar@{=>}[dll]^(.30){\dagger} \ar@{<=}[l] \ar@{<=}[dd] \ar@{=>}[dr]^{\dagger} & \\
    \mathbf{FM}  &&& \mathbf{DF} \\
    & \mathbf{BE} \ar@{=>}[uu]|(.30){\times}& \mathbf{BF} &
    }
\end{equation}
\end{theorem} 

The fact that $\LL$-equivalent cubic fourfolds are Fourier--Mukai partners was recently independently proved in \cite{BB2025} for very general cubic fourfolds in $\mathcal{C}_d$ for $d \equiv 0,2\pmod 6$ not divisible by $9$. We note that such cubic fourfolds satisfy Derived Torelli. They also prove \textbf{LE} $\implies$ \textbf{LF} and the fact that \textbf{BE} does not imply \textbf{LE}.

For the convenience of the reader, let us complete Diagram \ref{eq: our implications diagram} with the implications that are already known. The implication \textbf{FM} $\implies$ \textbf{DF} was proved recently in \cite{KS25}, and the fact that \textbf{BF} $\implies$ \textbf{DF} follows from the proof of the $D$-equivalence conjecture for hyperk\"ahler manifolds of \Kn-type of \cite{MSYZ25}. Finally, if the cubic fourfolds satisfy Derived Torelli, we also have \textbf{DF} $\implies$ \textbf{FM} by \cite[Corollary 9.3]{Bec22}. The question of whether implication \textbf{FM} $\implies$ \textbf{BE} holds is a conjecture by Huybrechts \cite[Conjecture 7.3.21]{Huy2023}, and the implication \textbf{BF} $\implies$ \textbf{BE} is conjectured in \cite[Conjecture 1.2]{BFM24}.
For a survey on these equivalence relations, we refer to \cite{BFM24}.

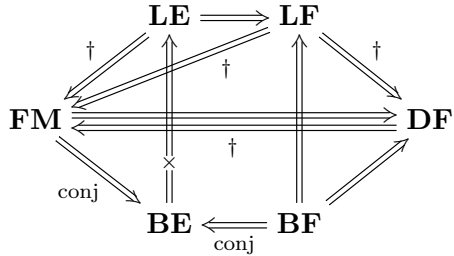
\begin{equation*}
    \xymatrix{
    & \mathbf{LE} \ar@{=>}[dl]_{\dagger}& \mathbf{LF} \ar@{=>}[dll]^(.30){\dagger} \ar@{<=}[l] \ar@{<=}[dd] \ar@{=>}[dr]^{\dagger} & \\
    \mathbf{FM} \ar@<0.5ex>@{=>}[rrr] \ar@{=>}[dr]_{\text{conj}} &&& \mathbf{DF} \ar@<0.5ex>@{=>}[lll]^{\dagger}\\
    & \mathbf{BE} \ar@{=>}[uu]|(.30){\times} & \mathbf{BF} \ar@{=>}[ur] \ar@{=>}[l]^{\text{conj}} &
    }
\end{equation*}

Finally, we prove that $\LL$-equivalence classes of cubic fourfolds are finite:

\begin{theorem}[Theorem \ref{thm: L equivalences classes are finite}] \label{thm: L equivalence classes are finite INTRO}
    For any cubic fourfold $X$, there exist finitely many cubic fourfolds $Y$ such that $X\sim_\LL Y$.
\end{theorem}

The ancillary files include the OSCAR code used for the lattice computations of this paper.

\subsection*{Acknowledgements}
We thank Franco Giovenzana for pointing out the paper \cite{FL23} to us, and we would also like to thank Evgeny Shinder for telling us about Proposition \ref{prop: birational implies same motive}. Finally, we thank Corey Brooke, Sarah Frei, and Lisa Marquand for useful comments on an earlier version of this paper. We also thank an anonymous referee for spotting a mistake in an earlier version of Proposition \ref{prop: L29 examples}. This work is a contribution to Project A22 in SFB 195 No. 286237555 of DFG.
The second author is a member of the INdAM group GNSAGA.

\section{Preliminaries}
\subsection{Lattice theory}
Our main reference for lattice theory is \cite{Nik80}.
A \textit{lattice} is a free, finitely generated abelian group $L$ together with a non-degenerate symmetric integral bilinear form $$(-,-)\colon L\times L\to \Z.$$

For elements $v,w\in L$, we often write $v^2\coloneqq (v, v)$ and $v\cdot w \coloneqq (v,w)$. A lattice $L$ is called \textit{even} if $v^2$ is even for all $v\in L$.

\begin{definition}
    Let $v\in L$ be an element of a lattice. The \textit{divisibility} of $v$ in $L$ is defined to be the positive integer
    \[
    \dv(v)\coloneqq \gcd_{w\in L}(v\cdot w).
    \]
\end{definition}

The \textit{dual} of $L$ is the group $L^\vee \coloneqq \Hom(L,\Z)$, and we have a natural group homomorphism
\[
\functionstar{L}{L^\vee}{v}{(v,-),}
\]
which is injective as $L$ is non-degenerate. The cokernel $A_L\coloneqq L^\vee/L$ is called the \textit{discriminant group}.
Via the embedding $L\hookrightarrow L\otimes \Q$, there is a natural identification
\begin{equation}\label{eq: dual lattice in rationalisation}
L^\vee = \left\{x\in L\otimes \Q \mid \forall v\in L: (x,v)\in \Z\right\}.
\end{equation}
This defines a natural symmetric bilinear form on $L^\vee$ taking values in $\Q$, which induces a bilinear form $b\colon A_L\times A_L\to \Q/\Z$. If $L$ is even, the bilinear form $b$ induces a quadratic form $q\colon A_L\to \Q/2\Z$. The pair $(A_L,q)$ is called a \textit{finite quadratic module.}

The order of $A_L$ is called the \textit{discriminant} $\disc(L)$ of $L$. If $\disc(L) = 1$, we say that $L$ is \textit{unimodular}. If $M\in \GL_n(\Z)$ is the matrix for the bilinear form on $L$ with respect to some $\Z$-basis of $L$, we have
\[
\disc(L) = |\det(M)|.
\]

If $L$ is a lattice, then $L\otimes \Q$ (resp. $L\otimes \mathbb{R}$) is a $\Q$- (resp. $\mathbb{R}$-)vector space which inherits a symmetric bilinear form from $L$. The \textit{signature} $\sgn(L) = (\ell_+,\ell_-)$ is defined to be the signature of the real symmetric bilinear form on the $\mathbb{R}$-vector space $L\otimes \mathbb{R}$. 

The \textit{genus} $\mathcal{G}(N)$ of a lattice $N$ is the set of isometry classes of lattices $N'$ satisfying each of the following:
\[
\rk(N) = \rk(N'), \qquad \sgn(N) = \sgn(N'), \qquad A_N \simeq A_{N'}.
\]

For lattices $L$ and $N$, a group homomorphism $f\colon N\to L$ with the property that $(v,w) = (f(v),f(w))$ for all $v,w\in N$ is called a \textit{metric morphism}. Bijective metric morphisms are called \textit{isometries}.
The group of isometries of $L$ is denoted $O(L)$. Similarly, we write $O(A_L)$ for the group of isometries of the finite quadratic module $A_L$. Note that we have a natural group homomorphism 
\[
\functionstar{O(L)}{O(A_L)}{f}{\overline{f}.}
\]

Metric morphisms are injective, since we assume our lattices to be non-degenerate.
A metric morphism is called a \textit{primitive embedding} if $L/N$ is a torsion-free abelian group. On the other hand, $f\colon N\hookrightarrow L$ is called an \textit{overlattice} if $L/N$ is a finite abelian group.
Two overlattices $f\colon N\hookrightarrow L$, $g\colon N'\hookrightarrow L'$ are called \textit{isomorphic}
if there exist isometries $\sigma\colon N\simeq N'$, $\tau\colon L\simeq L'$ fitting into a commutative square
\begin{equation*}
        \xymatrix{
            N \ar[d]_f \ar[r]_\simeq^\sigma& N \ar[d]^g\\
            L \ar[r]_\simeq^\tau& L'.            
        }
\end{equation*}
Let $f\colon N\hookrightarrow L$ be an overlattice, and assume $N$ and $L$ are both even. The following results also have analogues for odd lattices \cite{Nik80}.
From \eqref{eq: dual lattice in rationalisation}, we obtain a natural chain of finite-index injective group homomorphisms:
\[
N\hookrightarrow L\hookrightarrow L^\vee \hookrightarrow N^\vee.
\]
In particular, $G_L\coloneqq L/N$ can naturally be seen as a subgroup of $A_N$. Note that $L/N$ is an isotropic subgroup, and that $G_L^\perp = L^\vee/N$. This means that we have an isometry of finite quadratic modules $G_L^\perp/G_L\simeq A_L$. 

\begin{lemma}\cite[Proposition 1.4.2]{Nik80} \label{lem: overlattice rosetta stone}
    Let $f\colon N\to L$ and $g\colon N\to L'$ be two even overlattices of an even lattice $N$.
    \begin{enumerate}
        \item There is an isometry $\tau\colon L\simeq L'$ fitting into a commutative triangle
        \begin{equation*}
            \xymatrix{
            &N \ar[dl]_f \ar[dr]^g& \\
            L \ar[rr]^\tau_\simeq&& L'            
            }
        \end{equation*}
        if and only if $G_L = G_{L'}$ as subgroups of $A_N$.
        \item Given an isometry $\sigma\colon N\simeq N$, there exists an isometry $\tau\colon L\simeq L'$ fitting into a commutative square 
        \begin{equation*}
            \xymatrix{
            N \ar[d]_f \ar[r]_\simeq^\sigma& N \ar[d]^g\\
            L \ar[r]_\simeq^\tau& L'            
            }
        \end{equation*}
        if and only if $\overline{\sigma}(G_L) = G_{L'}$.
        \item The assignment $(f\colon N\hookrightarrow L)\mapsto G_L$ is a bijection between the set of isomorphism classes of overlattices of $N$, and the set of $O(N)$-orbits of isotropic subgroups of $A_N$. 
    \end{enumerate}
\end{lemma}

Overlattices play a crucial role in the study of primitive embeddings of lattices, for the following reason. If $f\colon N\to L$ is a primitive embedding, let us denote its \textit{orthogonal complement} by $$T = N^\perp \coloneqq \left\{v\in L\mid\forall n\in N: v\cdot n = 0\right\}\subset L.$$ Then $T\hookrightarrow L$ is also a primitive embedding. However, the metric morphism $N\oplus T\hookrightarrow L$ is an overlattice.
By the above discussion, we obtain a natural chain of finite-index embeddings
\[
N\oplus T \hookrightarrow L\hookrightarrow L^\vee \hookrightarrow N^\vee \oplus T^\vee.
\]
Thus the finite group $G_L\coloneqq L/(N\oplus T)$ is naturally a subgroup of $A_N\oplus A_T$, and we call it the \textit{gluing group} of the primitive embedding $f$.
The natural projection maps $A_N\oplus A_T\to A_N$ and $A_N\oplus A_T\to A_T$ restrict to injective group homomorphisms $i_N\colon G_L \hookrightarrow A_N$ and $i_T\colon G_L \hookrightarrow A_T$, respectively.

To summarise, the primitive embedding $f\colon N\hookrightarrow L$ determines a pair
$(G_f,\gamma_f),$ where $G_f \coloneqq i_N(G_L)\subset A_N$ is a subgroup, and $\gamma_f \coloneqq i_T\circ i_N^{-1} \colon G_L\to A_T$ is an injective group homomorphism, and the graph $\Gamma_{\gamma_f}\subset A_N\oplus A_T$ is precisely the gluing group $G_L$, which is an isotropic subgroup that satisfies $G_L^\perp/G_L \simeq A_L$.

\begin{remark} \label{rem: pair inducing a primitive embedding}
    Conversely, suppose $(G,\gamma)$ is a pair consisting of a subgroup $G\subset A_N$ and an injective group homomorphism $\gamma\colon G\hookrightarrow A_T$ with the property that the graph $\Gamma_\gamma\subset A_N\oplus A_T$ is an isotropic subgroup which satisfies $\Gamma_\gamma^\perp/\Gamma_\gamma \simeq A_L$.
    Then we obtain an overlattice $N\oplus T\hookrightarrow L'$ corresponding to $\Gamma_\gamma$ via Lemma \ref{lem: overlattice rosetta stone}. Moreover, the metric morphism $N\hookrightarrow N\oplus T\hookrightarrow L'$ is a primitive embedding with orthogonal complement equal to the image of $T$ in $L'$.
\end{remark}

By construction, the lattice $L'$ satisfies $A_{L'} \simeq \Gamma_\gamma^\perp/\Gamma_\gamma \simeq A_L$. 
Unfortunately, $L'$ is not necessarily isometric to $L$. However, since $L$ and $L'$ have the same rank, signature and discriminant group, they are in the same genus.

\begin{lemma}\cite[Proposition 1.5.1]{Nik80} \label{lem: primitive embeddings rosetta stone}
    Let $N$ be an even lattice and let $f\colon N\hookrightarrow L$ and $g\colon N\hookrightarrow L'$ be primitive embeddings into even lattices $L$, $L'$. Write $T\coloneqq N^\perp\subset L$ and $T'\coloneqq N^\perp\subset L'$. 
    \begin{enumerate}
        \item There is an isometry $\tau\colon L\simeq L'$ fitting into a commutative triangle
        \begin{equation*}
            \xymatrix{
                & N \ar[dl]_{f} \ar[dr]^g& \\
                L \ar[rr]^\tau_\simeq&& L'
            }
        \end{equation*}
        if and only if there is an isometry $\sigma_T\colon T\simeq T'$ such that $(G_f,\overline\sigma_T\circ \gamma_f) = (G_g,\gamma_g)$.
        \item Fix an isometry $\sigma_N\colon N\simeq N$. There is an isometry $\tau\colon L\simeq L'$ fitting into a commutative square
        \begin{equation*}
            \xymatrix{
                N \ar[d]_{f} \ar[r]^{\sigma_N}_\simeq& N \ar[d]^g \\
                L \ar[r]^\tau_\simeq& L'
            }
        \end{equation*}
        if and only if there is an isometry $\sigma_T\colon T\simeq T'$ such that the diagram
        \begin{equation*}
            \xymatrix{
                G_f \ar[r]^{\gamma_f} \ar[d]_{\overline\sigma_N}& A_T \ar[d]^{\overline\sigma_T} \\
                G_g \ar[r]^{\gamma_g}& A_{T'}
            }
        \end{equation*}
        commutes.
    \end{enumerate}
\end{lemma}

The following lemma was noted in \cite[\S 14.0]{Huy16}, and was a crucial part of the technical computations in \cite{Mei25}.
\begin{lemma}\label{lem: gluing group orders}
    Let $f\colon N\hookrightarrow L$ be a primitive embedding with orthogonal complement $T = N^\perp\subset L$ and gluing group $G\subset A_N\oplus A_T$. Then we have
    \[\disc(T)\disc(N)= |G|^2\disc(L).\]
\end{lemma}
\begin{proof}
    Recall that we have the following chain of embeddings 
    \[
    N\oplus T\hookrightarrow L \hookrightarrow L^\vee \hookrightarrow N^\vee \oplus T^\vee.
    \]
    We have $[L:N\oplus T] = |G|$ by construction, hence also $[N^\vee \oplus T^\vee : L^\vee] = |G|$. The embedding $L\hookrightarrow L^\vee$ has index $\disc(L)$, by definition. In particular, the index of the composition of the three embeddings is $[N^\vee \oplus T^\vee : N\oplus T] = |G|^2\disc(L)$. On the other hand, we have 
    \[
    \frac{N^\vee \oplus T^\vee}{N\oplus T} \simeq A_{N\oplus T}\simeq A_N\oplus A_T,
    \]
    hence $[N^\vee \oplus T^\vee: N\oplus T] = \disc(N)\disc(T).$ To summarise, we have
    \[
    \disc(N)\disc(T) = [N^\vee \oplus T^\vee: N\oplus T] = |G|^2\disc(L),
    \]
    as required.
\end{proof}

If $L$ is a lattice of rank $1$, and $v\in L$ is a generator with $v^2 = n\in \Z$, then we usually denote $L$ by $\langle n \rangle$. 
Another important example of a lattice is the root lattice $A_2$, which is the lattice of rank $2$ given by the intersection matrix
\[
A_2 \coloneqq 
\left(
\begin{matrix}
    2 & -1 \\
    -1 & 2
\end{matrix}
\right).
\]

Finally, the \textit{hyperbolic plane}, denoted $U$, is the indefinite even unimodular lattice of rank 2. Its intersection matrix is 
\[
U\coloneqq 
\left(
\begin{matrix}
    0 & 1 \\
     1 & 0
\end{matrix}
\right).
\]

\begin{remark}\label{rem: quadratic modules of order 3}
    Let $q\colon C_3 \to \Q/2\Z$ be a non-degenerate finite quadratic form of order $3$. Then if $x\in C_3$ is a generator, we see that $q(x) = \frac{a}{3}\pmod{2\Z}$, where $a\in (\Z/3\Z)^\times$. Therefore, $q(2x) = \frac{4a}{3} \equiv \frac{a}{3}\pmod{2\Z}.$ 
    As a consequence, $C_3$ and $C_3(-1)$ are the only two isomorphism classes of non-degenerate finite quadratic forms of order $3$. 

    We now compute the signatures of these finite quadratic forms. Suppose $C_3$ is the finite quadratic module of order $3$ with $q(x) = \frac{2}{3}\pmod{2\Z}$ for $x\in C_3$ a generator. We can find an explicit lattice whose discriminant group is $C_3$. Indeed, the lattice $A_2$ is a positive-definite lattice of rank $2$ and discriminant $3$, and we can compute that, for a generator $y\in A_{A_2}$, we have $q(y) = \frac{2}{3} \pmod{2\Z}$. In other words, the discriminant group of $A_2$ is precisely $C_3$. Therefore, we have 
    \[
    \sgn(C_3) \equiv \sgn(A_2) \equiv 2 \pmod 8,
    \] and also $\sgn(C_3(-1)) \equiv \sgn(A_2(-1)) \equiv -2 \equiv 6 \pmod 8$.
\end{remark}

\subsection{Cubic fourfolds and K3 surfaces}
Our main references for K3 surfaces and cubic fourfolds are \cite{Huy16,Huy2023}.
Let $X$ be a cubic fourfold. We consider the cohomology group $H^4(X,\Z)$. The intersection product defines a lattice structure on $H^4(X,\Z)$, which abstractly makes it isometric to the lattice
\[
\fullcublat \coloneqq E_8^{\oplus 2}\oplus U^{\oplus 2}\oplus \langle 1 \rangle^{\oplus 3}.
\]

Here, $E_8$ denotes the unique positive-definite unimodular lattice of rank $8$. The signature of $\fullcublat$ is $(21,2)$. We denote $h^2\coloneqq (1,1,1)\in \langle 1\rangle^{\oplus3}$. Then we define 
\[
\cublat \coloneqq (h^2)^\perp \simeq E_8^{\oplus 2}\oplus U^{\oplus 2}\oplus A_2.
\]
If $H^2\in H^4(X,\Z)$ is the square of a hyperplane, then the primitive cohomology $H^4(X,\Z)_\text{prim}\coloneqq (H^2)^\perp\subset H^4(X,\Z)$ is isometric to $\cublat$ as a lattice. Note that $\fullcublat$ is not an even lattice, but $\cublat$ is. 

The Hodge structure on $H^4(X,\Z)$ is of \textit{K3-type}. That is, 
\[
\begin{array}{lll}
     H^{4,0}(X) &\simeq H^{0,4}(X)&= 0  \\
     H^{3,1}(X) &\simeq H^{1,3}(X) &\simeq \CC. 
\end{array}
\]
Moreover, since $H^2$ is of type $(2,2)$, the primitive cohomology $H^4(X,\Z)_\text{prim}$ inherits a Hodge structure from $H^4(X,\Z)$ which is also of K3-type.

\begin{definition}
    A lattice with a Hodge structure is called a \textit{Hodge lattice}. If $H_1, H_2$ are Hodge lattices, then a \textit{Hodge isometry} is an isomorphism of groups $f\colon H_1\simeq H_2$ which is simultaneously an isomorphism of Hodge structures and an isometry.
\end{definition}

The algebraic part of $H^4(X,\Z)$ is denoted 
\[
A(X) \coloneqq H^{2,2}(X)\cap H^4(X,\Z).
\]
This is a positive-definite lattice, which contains the class $H^2$. Since the self-intersection number of $H^2$ is $3$, it follows that $A(X)$ is not an even lattice. Finally, we denote 
\[
A(X)_\text{prim} \coloneqq (H^2)^\perp \subset A(X).
\]
As opposed to the full algebraic lattice, the primitive part $A(X)_\text{prim}$ is an even lattice, since it is a sublattice of the primitive cohomology group $H^4(X,\Z)_\mathrm{prim}$, which is even. 

\begin{definition} \label{defn:transcendental lattice of cubic 4fold}
    The \textit{transcendental lattice} of a cubic fourfold is the Hodge lattice $T(X)\coloneqq A(X)^\perp \subset H^4(X,\Z)$. Equivalently, we may define $T(X) = A(X)_\text{prim}^\perp \subset H^4(X,\Z)_\text{prim}$.
\end{definition}

\begin{theorem}[Torelli Theorem for cubic fourfolds] \cite{Voisin1986} \label{thm: torelli for cubic fourfolds}
    Let $X$ and $Y$ be two cubic fourfolds. Then $X$ and $Y$ are isomorphic if and only if there exists a Hodge isometry $H^4(X,\Z)_\text{prim}\simeq H^4(Y,\Z)_\text{prim}$.
\end{theorem}

A remarkable class of divisors in the moduli space of cubic fourfolds are the Noether--Lefschetz divisors $\calC_d$, consisting of so called \emph{special cubic fourfolds}, see \cite{Has00}. A cubic fourfold $X$ is called special if its algebraic lattice $A(X)$ has rank at least $2$, or equivalently if $X$ contains a surface which is not homologous
to a complete intersection. As a consequence we have a primitive saturated lattice of rank $2$ containing the square of the hyperplane class, and the class $\alpha$ of this surface:
$$K = \langle h^2, \alpha \rangle \subset A(X).$$
The locus of special cubic fourfolds with $\disc(K)=d$ is denoted by $\calC_d$.

A fundamental tool for studying cubic fourfolds is the period map, which associates a cubic fourfold $X$ in its moduli space to its primitive cohomology $H^4_{\prim}(X,\bbZ)$, viewed as a point in the corresponding period domain, see \cite{Huy2023} for more details.

\begin{theorem}[Surjectivity of the period map for cubic fourfolds] \cite{Looijenga2009, Laza2010} \label{thm: surjectivity of the period map}
The image of the period map for cubic fourfolds is the complement of the divisors $\calC_2 \cup \calC_6$.
\end{theorem}

It follows from the surjectivity of the period map that $A(X)_\mathrm{prim}$ contains no vectors of square $2$, nor any vectors of square $6$ and divisibility $3$ in $H^4(X,\Z)_\mathrm{prim}$.
In the situation considered below, a vector of square $6$ and divisibility $3$ or $6$ in the primitive algebraic lattice may have divisibility $3$ in the primitive cubic lattice. In the proof of Lemma \ref{lem: genus is fixed}, we will show that such a vector indeed always has divisibility $3$ in the primitive cubic lattice, and therefore we exclude both possibilities in the definition below. 
\begin{definition}
    Let $N$ be a lattice. We denote 
    \[
    \mathcal{H}(N):= \left\{N'\in \mathcal{G}(N)\mid \nexists v\in N' : v^2 = 2 \text{ or } v^2 = 6 \wedge 3\mid\dv_{N'}(v)\right\}.
    \]
\end{definition}

The derived category of a cubic fourfold $X$ admits a semiorthogonal decomposition \cite{Kuz2010}:
\[
\Db(X)\simeq \langle \mathcal{A}_X,\bigo_X, \bigo_X(1), \bigo_X(2)\rangle.
\]
The component $\mathcal{A}_X$ is a K3 category called the \textit{Kuznetsov component}. The topological Grothendieck group $K_\text{top}(\mathcal{A}_X)$ has a natural Hodge structure of K3-type of weight $2$ as defined in \cite{AT2014}. The resulting Hodge lattice is usually denoted $\tilde{H}(\mathcal{A}_X,\Z)$. As a lattice, $\tilde{H}(\mathcal{A}_X,\Z)$ is isometric to the even, unimodular lattice of signature $(20,4)$:
\[
\tilde{\Lambda} \coloneqq U^{\oplus 4}\oplus E_8^{\oplus 2}.
\]
The transcendental lattice $T(X)$ has a natural embedding into $\tilde{H}(\mathcal{A}_X,\Z)$. In fact, $T(X)$ is the \textit{transcendental sublattice} of $\tilde{H}(\mathcal{A}_X,\Z)$ which is a notion that we will define in Section \ref{sec: Hodge realisation}.

We now turn our attention to K3 surfaces. We assume all our K3 surfaces to be projective. If $S$ is a K3 surface, the intersection pairing and Hodge structure turn the cohomology group $H^2(S,\Z)$ into a Hodge lattice of K3-type. The underlying lattice structure of this Hodge lattice is 
\[
\klat \coloneqq U^{\oplus 3}\oplus E_8(-1)^{\oplus 2}.
\]
This is an even unimodular lattice of signature $(3,19)$. The algebraic part of $H^2(S,\Z)$ is called the N\'eron--Severi lattice of $S$, denoted 
\[
    \NS(S) \coloneqq H^{1,1}(S)\cap H^2(S,\Z).
\]
The N\'eron--Severi lattice has signature $(1,\rho-1)$, where $\rho = \rk\NS(S)$ is the \textit{Picard rank} of $S$.
Its orthogonal complement is the \textit{transcendental lattice} 
\[
T(S) \coloneqq \NS(S)^\perp\subset H^2(S,\Z).
\]

Similarly to cubic fourfolds, the Torelli Theorem for K3 surfaces \cite[Theorem 7.5]{Huy16} states that two K3 surfaces $S$ and $S'$ are isomorphic if and only if the cohomology groups $H^2(S,\Z)$ and $H^2(S',\Z)$ are Hodge isometric.

For a K3 surface $S$, we consider its full cohomology group $H^*(S,\Z) \simeq H^0(S,\Z) \oplus H^2(S,\Z) \oplus H^4(S,\Z)$. Via the natural isomorphisms $H^0(S,\Z) \simeq \Z$ and $H^4(S,\Z) \simeq \Z$, we usually write an element of $H^*(S,\Z)$ as $(r,\ell,s)$, where $r,s\in \Z$ and $\ell\in H^2(S,\Z)$. We can put a lattice structure on $H^*(S,\Z)$ by the following definition:
\[
(r,\ell,s)\cdot (r',\ell',s') = -2(rs'+r's) + \ell\cdot \ell'.
\]
In other words, we declare that $ H^0(S,\Z) \oplus H^4(S,\Z)$ gets the structure of a hyperbolic plane $U$ which is orthogonal to $H^2(S,\Z)$. 
There is also a natural Hodge structure on $H^*(S,\Z)$ which it inherits from $H^2(S,\Z)$. With this Hodge structure and lattice structure, we denote this Hodge lattice by $\tilde{H}(S,\Z)$ and call it the \textit{Mukai lattice} of $S$. 

We recall the Derived Torelli Theorem by Mukai and Orlov, which is the most fundamental result regarding derived equivalence for K3 surfaces:
\begin{theorem}[Derived Torelli Theorem for K3 surfaces] \cite{Muk87,Orl03} \label{thm: derived torelli for k3s}
    Let $S$ and $S'$ be K3 surfaces. The following are equivalent:
    \begin{enumerate}
        \item $S$ and $S'$ are Fourier--Mukai partners, that is, there is an equivalence $\Db(S)\simeq \Db(S')$.
        \item There is a Hodge isometry $T(S)\simeq T(S')$.
        \item There is a Hodge isometry $\tilde{H}(S,\Z) \simeq \tilde{H}(S',\Z)$.
    \end{enumerate}
\end{theorem}

Returning to cubic fourfolds, it is known for several families of special cubic fourfolds that the Kuznetsov component is explicitly realised as the derived category of a (possibly twisted) K3 surface: for cubic fourfolds containing a plane this goes back to \cite{Kuz2010} and has been extended to the non-generic nodal case in \cite{Mos2018}. Recently this has been studied for cubic fourfolds containing multiple planes in \cite{Hart2025}. However, it is currently an open conjecture that Theorem \ref{thm: derived torelli for k3s} has an analogue for the Kuznetsov components of cubic fourfolds.
That is, for two cubic fourfolds $X, Y$, it is currently not known whether 
there is an equivalence $\mathcal{A}_X\simeq \mathcal{A}_Y$ if there is a
Hodge isometry $\tilde{H}(\mathcal{A}_X,\Z) \simeq \tilde{H}(\mathcal{A}_Y,\Z)$. There is the following positive result, however:

\begin{theorem}  \cite{HS05,AT2014} \label{thm: known derived torelli statements}
    Let $X$ and $Y$ be cubic fourfolds. Consider the following three statements:
    \begin{enumerate}
        \item There is a Fourier--Mukai equivalence  $\mathcal{A}_X\simeq \mathcal{A}_Y$.
        \item There is a Hodge isometry $\tilde{H}(\mathcal{A}_X,\Z) \simeq \tilde{H}(\mathcal{A}_Y,\Z)$.
        \item There is a Hodge isometry $T(X)\simeq T(Y)$.
    \end{enumerate}
    Then we have 
    \[
    (1) \implies (2) \implies (3).
    \]
    Moreover, suppose that there is a twisted K3 surface $(S,\alpha)$ and a Hodge isometry $\tilde{H}(\mathcal{A}_X,\Z)(-1) \simeq \tilde{H}(S,\alpha,\Z)$. Then in addition, we have $(2)\implies (1).$
\end{theorem}
\subsection{Hodge realisation}
\label{sec: Hodge realisation}
We briefly recall the constructions of \cite{Efi18,Mei24,Mei25}.
We denote by $\operatorname{HS}_\Z$ the category of pure, graded, polarisable, integral Hodge structures. The \textit{Grothendieck group of Hodge structures} is defined to be the group 
\[
\gghsf \coloneqq \Z[\operatorname{HS}_\Z/_\simeq]\big/([H_1\oplus H_2] - [H_1] - [H_2]).
\]
There is a \textit{Hodge realisation} map 
\[
\operatorname{Hdg}_\Z \colon \gravc \to \gghsf
\]
defined by sending the class of a smooth, projective variety $[X]$ to the class of its full cohomology group $[H^\bullet(X,\Z)]$, considered with its Hodge structure.

\begin{remark}
    The Hodge realisation map factors through $\gravc[\LL^{-1}]$. In other words, if $X$ and $Y$ are smooth, projective, $\LL$-equivalent varieties, then we have $[H^\bullet(X,\Z)] = [H^\bullet(Y,\Z)]$ in $\gghsf$. 
\end{remark}

For any $n\in \Z$, we have a full subcategory $\operatorname{HS}_{\Z,n}$ of $\operatorname{HS}_\Z$, consisting of pure, polarisable, integral Hodge structures of weight $n$. This induces group homomorphisms for each $n$:
\[
\functionstar{\gghsf}{\gghsn}{\text{[}H^\bullet]}{[H^n].}
\]

Suppose we have $H\in \operatorname{HS}_{\Z,2n}$ a Hodge structure of weight $2n$, then the Tate twist $H(n-1)$ is a Hodge structure of weight $2$. This yields an isomorphism 
\[
K_0^\oplus(\operatorname{HS}_{\Z,2n}) \simeq \gghs.
\]

For a Hodge structure $H\in \operatorname{HS}_{\Z,2}$ of weight $2$, we denote by $T(H)$ the transcendental sub-Hodge structure of $H$. That is, $T(H)$ is the minimal, integral, primitive sub-Hodge structure with the property that $H^{2,0}\subset T(H)_\CC$. Moreover, we write
\[
\NS(H)\coloneqq H^{1,1}\cap H. 
\]

\begin{example}
    If $X$ is a hyperk\"ahler manifold, then $$\NS(H^2(X,\Z)) = \NS(X), \qquad T(H^2(X,\Z)) = T(X).$$
    If $Y$ is a cubic fourfold, then $\NS(H^4(Y,\Z)(1))=A(Y)(1)$ and $T(H^4(Y,\Z)(1))=T(Y)(1)$.
\end{example}

\begin{definition}
    For $H\in \operatorname{HS}_{\Z,2}$, we call the group 
    \[
    G(H)\coloneqq \frac{H}{\NS(H)\oplus T(H)}
    \]
    the \textit{gluing group} of $H$.
\end{definition}

\begin{lemma}\cite{Mei24,Mei25} \label{lem: gluing group homomorphism}
    There is a group homomorphism 
    \[
    \functionstar{\gghs}{\Q^\times}{\normalfont{[}H]}{\left|G(H)_{\text{tors}}\right|.}
    \]
\end{lemma}

\begin{lemma} \label{lem: L equivalent cubic fourfolds have equal discriminants}
    If $X$ is a cubic fourfold, then we have 
    \[
    |G(H^4(X,\Z)(1))| = \disc(T(X)).
    \]
\end{lemma}
\begin{proof}
    The transcendental lattice $T(X)$ is a primitive sublattice of the unimodular lattice $H^4(X,\Z)$, hence the gluing group is isomorphic to $A_{T(X)}$ by Remark \ref{rem: pair inducing a primitive embedding}.
\end{proof}

\begin{lemma}\label{lem: coarseness of a fano variety}
    Let $F(X)$ be the Fano variety of lines on a cubic fourfold $X$. Then we have 
    \[
    |G(H^2(F(X),\Z))| = \disc(T(F(X))) = \disc(T(X)).
    \]
\end{lemma}
\begin{proof}
    Addington showed in \cite[Corollary 8]{Add16} that there is a natural embedding $A_2 = \langle \lambda_1,\lambda_2\rangle\subset \tilde{H}(\mathcal{A}_X,\Z)$ such that 
    \begin{equation}\label{eq: natural embedding of H2 of the fano}
        H^2(F(X),\Z) \simeq (\lambda_1^\perp)(-1) \subset \tilde{H}(\mathcal{A}_X,\Z)(-1), 
    \end{equation}
    where the $(-1)$ indicates that we rescaled the bilinear forms by $-1$. Since the divisibility of $\lambda_1$ in $\tilde{H}(\mathcal{A}_X,\Z)$ is $1$, it follows that the \textit{coarseness} of $F(X)$ is $\crs(F(X)) = 1$ (see \cite[Definition 2.4]{Mei25} for the definition of coarseness). 
    Therefore, by \cite[Theorem 3.5]{Mei25}, we find 
    \[
    |G(H^2(F(X),\Z))| = \frac{\disc(T(F(X)))}{\crs(F(X))} = \disc(T(F(X))),
    \]
    as required. Moreover, by \eqref{eq: natural embedding of H2 of the fano}, $T(F(X))$ is Hodge anti-isometric to the transcendental sublattice of $\tilde{H}(\mathcal{A}_X,\Z)$, which is $T(X)$, hence $\disc(T(F(X))) = \disc(T(X))$. 
\end{proof}

Suppose that there is a K3 surface $S$ with a Hodge isometry $T(X)(-1)\simeq T(S)$, so that $F(X)$ is a moduli space of stable sheaves on $S$ by \cite[Corollary 9.6]{Bec22}. Then Lemma \ref{lem: coarseness of a fano variety} is equivalent to saying that $F(X)$ is a fine moduli space of sheaves on $S$ by \cite[Theorem 3.9]{MM24}.

The following result appears as Lemma 3.2 in \cite{BB2025}.
\begin{theorem} \label{thm: L equivalence implies Hodge isometry}
    Let $X$ and $Y$ be cubic fourfolds (resp. K3 surfaces). Suppose $\rk(T(X))\neq 4$, and that $\End_{\mathrm{Hodge}}(T(X)) \simeq \Z$. Then, if $[T(X)]=[T(Y)]$ in $\gghsf$, there is a Hodge isometry $T(X)\simeq T(Y)$.  
\end{theorem}
\begin{proof}
    If $X$ and $Y$ are K3 surfaces, the statement is follows from \cite[Proposition 3.11]{Mei24} combined with \cite[Lemma 3.7]{Efi18}. 
    
    If $X$ and $Y$ are cubic fourfolds, we have $[T(X)(1)] = [T(Y)(1)]$ in $\gghs$. 
    Since  $\End_{\mathrm{Hodge}}(T(X)) \simeq \Z$, this implies by \cite[Lemma 3.7]{Efi18} that there exists an isomorphism of Hodge structures (not necessarily a Hodge isometry) $T(X)(1)\simeq T(Y)(1)$. 
    Moreover, since $X$ and $Y$ are $\LL$-equivalent, we have $\disc(T(X)) = \disc(T(Y))$ by Lemma \ref{lem: gluing group homomorphism} and Lemma \ref{lem: L equivalent cubic fourfolds have equal discriminants}. It now follows from \cite[Corollary 3.10]{Mei24} that we have a Hodge isometry $T(X)(1) \simeq T(Y)(1)$, which clearly induces a Hodge isometry $T(X)\simeq T(Y)$ via the Tate twist. 
\end{proof}

\section{Counting Fourier--Mukai partners of cubic fourfolds}

\label{sec: counting formula}
We compute a counting formula for Fourier--Mukai partners of cubic fourfolds with certain properties. The results in this section are closely related to \cite{HLOY02,Ma10,PertFMPartners21, FL23, BBM25}. 

We begin by collecting a few remarks and criteria concerning cubic fourfolds that satisfy Derived
Torelli in the sense of Definition \ref{def: derived torelli assumption}, and we describe how to find an example of a cubic which doesn't.

\begin{remark} \label{rem:1 implies 2 in the assumption}
Notice that the implication $(1) \Rightarrow (2)$ in Definition \ref{def: derived torelli assumption} always holds. By \cite[Proposition 3.4]{Huy17}, any equivalence $\mathcal{A}_X\simeq \mathcal{A}_Y$ induces an orientation-preserving Hodge isometry $\widetilde H(\calA_X,\bbZ) \to \widetilde H(\calA_Y,\bbZ)$ (\cite[Theorem 1.2]{Huy17}). Then, by \cite[Proposition 2.3]{AT2014}, there is an identification $H^4_{\mathrm{prim}}(X,\bbZ)\cong A_2^\perp\subset\widetilde H(\calA_X,\bbZ)$. By Voisin’s proof of the integral Hodge conjecture for cubic fourfolds, the primitive
algebraic lattices coincide with the primitive Hodge classes, $A(X)_{\mathrm{prim}}
 = H^4_{\mathrm{prim}}(X,\bbZ)\cap H^{2,2}(X)$. Hence we have a Hodge isometry $T(X)=A(X)_{\mathrm{prim}}^\perp \cong A(Y)_{\mathrm{prim}}^\perp = T(Y)$.
\end{remark}

The following result follows immediately from Theorem \ref{thm: known derived torelli statements}.
\begin{proposition} \label{prop: derived torelli criteria}
    In each of the following cases, the cubic fourfold $X$ satisfies Derived Torelli:
    \begin{enumerate}
        \item There is a twisted K3 surface $(S,\alpha)$ such that there exists a Hodge isometry $\tilde{H}(S,\alpha,\Z) \simeq \tilde{H}(\mathcal{A}_X,\Z)(-1)$, and there is a unique embedding $T(X)\hookrightarrow \tilde{\Lambda}$.
        \item If there exists a K3 surface $S$ and a Hodge isometry $T(X)\simeq T(S)(-1)$.
    \end{enumerate}
\end{proposition}

Note that (2) is a special case of (1) in Proposition \ref{prop: derived torelli criteria}.

\begin{remark} \label{cases in which the assumption doesn't hold}
Notice that a weaker Derived Torelli statement can be found in Conjecture 6.1 of \cite{Huy25},
where the existence of an equivalence $\calA_X\simeq \calA_Y$ is conjectured to be equivalent to an
orientation preserving Hodge isometry $\widetilde H(X,\Z)\simeq \widetilde H(Y,\Z)$. 
As Hodge isometries of $T(X)$ need not to lift to Hodge isometries of the Mukai lattice,
it is reasonable to expect that some cubic fourfolds do not satisfy Derived Torelli in the sense of
Definition \ref{def: derived torelli assumption}.
\end{remark}

\begin{remark}\label{rem: isometry actions}
    Let $T$ be a lattice such that $\disc(T)$ is not divisible by $3$, suppose that we have a primitive embedding $T\hookrightarrow \cublat,$ and let $N\coloneqq T^\perp\subset \cublat$.
    Let $G = \frac{\cublat}{T\oplus N}$ be the gluing group for this primitive embedding. Then, by Lemma \ref{lem: gluing group orders}, we have     
    \[\frac{\disc(T)}{|G|}\cdot \frac{\disc(N)}{|G|} = \disc(\cublat) = 3.\]
    Since $\disc(T)$ is not divisible by $3$, and $|G|$ divides both $\disc(T)$ and $\disc(N)$, it follows that we have 
    \[
    \disc(T) = |G|, \qquad \disc(N) = 3\cdot |G|.  
    \]
    In particular, the canonical embedding $G\hookrightarrow A_T$ is actually a group isomorphism $G\simeq A_T$. Therefore, the natural embedding $G\hookrightarrow A_N$ induces an isometry $A_N \simeq A_T(-1)\oplus C_3$, where $C_3$ is one of the two non-degenerate quadratic modules of order $3$ (see Remark \ref{rem: quadratic modules of order 3}).
    As a consequence, we have a canonical isomorphism
    \begin{equation}\label{eq: isometries of discriminant split}
        O(A_N) \simeq O\left(A_T(-1)\right)\oplus O(C_3) \simeq O\left(A_T\right)\oplus \left\{\pm \id_{C_3}\right\}.
    \end{equation}
    From this observation, we obtain an action of $O(N)$ on $O(A_T)$. Indeed, if $f\in O(N)$ is an isometry, it induces an isometry $\overline{f}\in O(A_N)$, which can be uniquely written as $\overline{f} = (\sigma, \tau)\in O(A_T)\oplus O(C_3)$ via \eqref{eq: isometries of discriminant split}.

    On the other hand, if $\disc(N)$ is not divisible by $3$, by the symmetry of the situation we obtain an action of $O(T)$ on $O(A_N)$.
\end{remark}

The main result in this section is the following counting formula:

\begin{theorem} \label{thm: counting formula}
    Let $X$ be a cubic fourfold that \emph{satisfies Derived Torelli} (See Definition \ref{def: derived torelli assumption}). Assume moreover that $\disc(T(X))$ is not divisible by $3$. Then we have 
    \begin{equation} \label{eq: counting formula}
    \#\FM(X) = \sum_{N'\in \mathcal{H}(N)}|O(N')\backslash O(A_{T(X)})/O_{\mathrm{Hodge}}(T(X))|.
    \end{equation}
    Here, $N = A(X)_{\text{prim}}$ and the action of $O(N')$ on $O(A_{T(X)})$ is the one defined by Remark \ref{rem: isometry actions}.
\end{theorem}

\begin{remark}
    If one replaces $\mathcal{H}(N)$ by $\mathcal{G}(N)$ in \eqref{eq: counting formula}, one obtains the number of \textit{virtual Fourier--Mukai partners} as in \cite{BBM25}.
\end{remark}

\begin{definition} \label{defn:Emb}
    For lattices $T,N,L$, we denote by $\Emb_N(T,L)$ the set of $O(\cublat)$-orbits of primitive embeddings $f\colon T\hookrightarrow L$ such that $T^\perp \simeq N$. Note that $O(T)$ acts naturally on $\Emb_N(T,L)$.
\end{definition}

We now prepare for the proof of Theorem \ref{thm: counting formula} with some technical results.

\begin{lemma}\label{lem: genus is fixed}
    Let $X$ be a cubic fourfold such that $\disc(T(X))$ is not divisible by $3$. 
    Then, for cubic fourfold $Y$ such that $T(Y)$ is Hodge isometric to $T(X)$, we have 
    \[A(Y)_{\text{prim}}\in \mathcal{H}\left(A(X)_{\text{prim}}\right).\]
    In particular, $A(Y)_{\text{prim}}$ is in the same genus as $A(X)_{\text{prim}}$. 
\end{lemma}
\begin{proof}
    To ease notation, we write $N = A(X)_{\text{prim}}$ and $N' = A(Y)_{\text{prim}}$.
    Since there is a Hodge isometry $T(X) \simeq T(Y)$, just as in Remark \ref{rem: isometry actions}, we find $A_N \simeq A_{T(X)}(-1)\oplus C_3$, where $C_3$ is one of the two non-degenerate quadratic modules of order $3$. Likewise, we obtain $A_{N'}\simeq A_{T(Y)}(-1)\oplus C_3',$ where $C_3'$ is one of the two non-degenerate quadratic modules of order $3$. By Remark \ref{rem: quadratic modules of order 3}, we see that $C_3\simeq C_3'$ or $C_3\simeq C_3'(-1)$.
    We now argue that we must have $C_3 \simeq C_3'$, which implies $A_N\simeq A_{N'}$, thus $N'$ is in the genus of $N$. 
    By construction, we have $\sgn(N) = \sgn(N')$, thus $\sgn(A_N) \equiv \sgn(A_{N'})\pmod 8$. This implies that we have 
    \[
    0=\sgn(A_N)-\sgn(A_{N'}) = \sgn(A_{T(X)(-1)}) + \sgn(C_3) - \sgn(A_{T(Y)(-1)})-\sgn(C_3').
    \]
    Since $T(X)$ is isometric to $T(Y)$, we have $\sgn(A_{T(X)(-1)}) = \sgn(A_{T(Y)(-1)})$, hence 
    \[\sgn(C_3) =\sgn(C_3').\]
    Recall from Remark \ref{rem: quadratic modules of order 3} that the signature of $C_3$ is either $2$ or $6$ modulo $8$. In particular, we find $\sgn(C_3(-1)) = -\sgn(C_3) \not \equiv \sgn(C_3)\pmod 8$. Since $\sgn(C_3) = \sgn(C_3')$, we obtain $C_3 \simeq C_3'$, which shows that $N$ and $N'$ are in the same genus. 

    We now check that $N' \in \mathcal{H}(N)$. 
    Firstly, note that if $N'$ were to contain a vector of square $2$, this would be a contradiction with the surjectivity of the period map, see Theorem \ref{thm: surjectivity of the period map}.
    
    Similarly, we now prove that $N'$ cannot contain a vector of square $6$ and divisibility $3$ or $6$. For the sake of contradiction, let us suppose that there exists a vector $v\in N'$ with $v^2=6$ and $3\mid\dv_{N'}(v)$. 
    Then the element $\frac{v}{3}\in A_{N'} \simeq A_{T(Y)(-1)}\oplus C_3$ has order $3$, hence it is a generator for $C_3$ and, importantly, is not contained in $A_{T(Y)(-1)}$. The gluing group $G \coloneqq \frac{H^4(Y,\Z)_\mathrm{prim}}{N\oplus T(Y)}$ is equal to $A_{T(Y)(-1)}\oplus A_{T(Y)}$ as a subgroup of $A_{N'} \oplus A_{T(Y)}$. 
    Therefore, we have a natural isomorphism $A_{H^4(Y,\Z)_\mathrm{prim}} \simeq G^\perp / G \simeq C_3$, showing that $A_{H^4(Y,\Z)_\mathrm{prim}}$ is also generated by $\frac{v}{3}$. This implies that $\frac{v}{3}\in H^4(Y,\Z)_\mathrm{prim}^\vee$, hence $v$ must have divisibility $3$ in $H^4(Y,\Z)_\mathrm{prim}$. This contradicts the surjectivity of the period map, c.f. Theorem \ref{thm: surjectivity of the period map}.

    This shows that $N'\in \mathcal{H}(N)$, as required.
\end{proof}

\begin{definition}
    For a cubic fourfold $X$ and a lattice $N$, we write $\FM(X,N)$ for the set of isomorphism classes of Fourier--Mukai partners $Y$ of $X$ such that there exists an isometry $A(Y)_\text{prim} \simeq N$.
\end{definition}

\begin{proposition} \label{prop: correspondence between FM partners and primitive embeddings}
    Let $X$ be a cubic fourfold which satisfies Derived Torelli. 
    Let $N'$ be a lattice. Then there is a one-to-one correspondence between the set $\FM(X,N')$ and the set of $O_{\mathrm{Hodge}}(T(X))$-orbits of primitive embeddings in $\Emb_{N'}(T(X),\cublat)$.
\end{proposition}
\begin{proof}
    Firstly, we note that, by Lemma \ref{lem: genus is fixed}, the sets $\FM(X,N')$ and $\Emb_{N'}(T(X),\cublat)$ are both empty if $N' \notin \mathcal H(A(X)_\mathrm{prim}). $
    Therefore we assume $N \in \mathcal H(A(X)_\mathrm{prim}).$
    
    If $Y\in \FM(X,N')$, and we fix any Hodge isometry $\psi\colon T(X)\simeq T(Y)$, and any isometry $\phi\colon H^4(Y,\Z)_\text{prim} \simeq \cublat$ then the composition 
    \[
        i\colon T(X)\overset{\psi}{\simeq} T(Y)\hookrightarrow H^4(Y,\Z)_\text{prim}\overset{\phi}{\simeq} \cublat
    \]
    is a primitive embedding with $T(X)^\perp \simeq N'$. Note that $i$ depends on the choice of a Hodge isometry $\psi$ and the choice of an isometry $\phi\colon H^4(Y,\Z)_\text{prim}\simeq \cublat$. However, if $\psi'\colon T(X)\simeq T(Y)$ is another Hodge isometry, and $\phi'\colon H^4(Y,\Z)_\text{prim}\simeq \cublat$ is another isometry, and we denote by $i'\colon T(X)\hookrightarrow \cublat$ the resulting primitive embedding, then the commutative diagram
    \begin{equation*}
        \xymatrix{
            i\colon &T(X)\ar[r]_\simeq^\psi \ar[d]_{(\psi')^{-1}\circ \psi} & T(Y) \ar@{^(->}[r] \ar@{=}[d] & H^4(Y,\Z)_\text{prim} \ar@{=}[d] \ar[r]_\simeq^\phi & \cublat \ar[d]_\simeq^{\phi'\circ \phi^{-1}} \\
            i'\colon &T(X)\ar[r]_\simeq^{\psi'} & T(Y) \ar@{^(->}[r] & H^4(Y,\Z)_\text{prim} \ar[r]_\simeq^{\phi'} & \cublat
        }
    \end{equation*}
    shows that the primitive embeddings $i$ and $i'$ are in the same $O_\mathrm{Hodge}(T(X))$-orbit. We will denote an embedding in this orbit by $i_Y$ from now on.
    This process thus yields a well-defined map:
    \[
    \function{\operatorname{emb}}{\FM(X,N')}{\Emb_{N'}(T(X),\cublat)\big/O_\mathrm{Hodge}(T(X))}{Y}{i_Y}
    \]
    The injectivity of this map follows essentially from the Torelli Theorem: if $Y,Y'\in \FM(X,N')$, such that $i_Y$ and $i_{Y'}$ are in the same $O_{\mathrm{Hodge}}(T(X))$-orbit, then there is a Hodge isometry $\psi\colon T(X)\simeq T(X)$ fitting into a commutative diagram:
    \[
    \xymatrix{
         i_Y\colon &T(X)\ar[r]^\simeq \ar[d]_{\psi} & T(Y) \ar@{^(->}[r] & H^4(Y,\Z)_\text{prim}  \ar[r]_\simeq^{\phi} & \cublat \ar@{=}[d] \\
            i_{Y'}\colon &T(X)\ar[r]^\simeq & T(Y') \ar@{^(->}[r] & H^4(Y',\Z)_\text{prim} \ar[r]_\simeq^{\phi'} & \cublat
        }
    \]
    Now $(\phi')^{-1}\circ \phi\colon H^4(Y,\Z)_\text{prim}\simeq H^4(Y',\Z)_\text{prim}$ is a Hodge isometry, so that $Y\simeq Y'$ by the Torelli Theorem.

    For surjectivity, note that any primitive embedding $i\colon T(X)\hookrightarrow \cublat$ induces a Hodge structure on $\cublat$. 
    We denote the resulting Hodge lattice $L_i$, i.e. $L_i$ is the Hodge lattice whose underlying lattice structure is $\cublat$ and whose Hodge structure is induced by $i$. 
    Since $T(X)^\perp \simeq N'\in \mathcal{H}(N)$, it follows that $N'$ contains no vectors of square $6$ and divisibility $3$ or $6$, hence there is a unique cubic fourfold $Y$ for which there is a Hodge isometry $H^4(Y,\Z)_\text{prim} \simeq L_i$.
    By construction, we have $T(Y)\simeq T(L_i(1)) \simeq T(X)$, and $A(Y)_\text{prim} \simeq \NS(L_i(1))\simeq N'$, where $T(L_i(1))$ and $\NS(L_i(1))$ are defined as in Section \ref{sec: Hodge realisation}. Since $X$ satisfies Derived Torelli, it follows that $Y\in \FM(X,N')$. Clearly, we have $i = \operatorname{emb}(Y)$. 
    Therefore, the map $\operatorname{emb}$ is surjective, as required.
\end{proof}

\begin{corollary} \label{cor: counting primitive embeddings}
    Let $X$ be a cubic fourfold which satisfies Derived Torelli, and such that $\disc(T(X))$ is not divisible by $3$. Then we have 
    \[
        \FM(X) = \bigsqcup_{N'\in \mathcal{H}(N)} \FM(X,N') \simeq \bigsqcup_{N' \in \mathcal{H}(N)} \Emb_{N'}(T(X),\cublat)/O_{\mathrm{Hodge}}(T(X)).
    \]
\end{corollary}

We are now ready to prove Theorem \ref{thm: counting formula}.

\begin{proof}[Proof of Theorem \ref{thm: counting formula}]
    Let $N' \in \mathcal{H}(N)$. Fix, once and for all, an isometry $A_{N'} \simeq A_{T(X)} \oplus C_3$. Via this isometry, we will view $A_{T(X)}$ as a quadratic submodule of $A_{N'}$. 

    Suppose $i\colon T(X)\hookrightarrow \cublat$ is a primitive embedding with orthogonal complement $T(X)^\perp \simeq N'$. From Lemma \ref{lem: primitive embeddings rosetta stone}, it follows that $i$ induces a pair $(G_i,\gamma_i)$, where $G_i\subset A_{T(X)}$ and $\gamma_i\colon G\hookrightarrow A_{N'}$. Here, $G_i$ is isomorphic, as a group, to the gluing group 
    \[
        G_i\simeq \frac{\cublat}{T(X) \oplus N'}.
    \]
    In particular, just as in Remark \ref{rem: isometry actions}, we find $G_i = A_{T(X)}$. Note, moreover, that $A_{T(X)}\subset A_{N'}$ is the unique subgroup of $A_{N'}$ order $\disc(T(X))$. Therefore, the embedding $\gamma_i: A_{T(X)}\hookrightarrow A_{N'}$ has image $A_{T(X)}$ and therefore induces an isometry $A_{T(X)}\simeq A_{T(X)}$, which we shall also denote by $\gamma_i$. 

    We claim that the following is a well-defined bijection:
    \[
    \isomorphism{\Gamma_{N'}}{\Emb_{N'}(T(X),\cublat)/O_{\mathrm{Hodge}}(T(X))}{O(N')\backslash O(A_{T(X)}) / O_{\mathrm{Hodge}}(T(X))}{i}{\gamma_i.}
    \]
    Let us begin by showing that $\Gamma$ is well-defined and injective. If $i,i'\colon T(X)\hookrightarrow \cublat$ are two primitive embeddings contained in $\Emb_{N'}(T(X),\cublat)$. Then, by Lemma \ref{lem: primitive embeddings rosetta stone}, $i$ and $i'$ are in the same $O_{\mathrm{Hodge}}(T(X))$-orbit, i.e. there is a commutative square
    \[
    \xymatrix{
        T(X) \ar@{^(->}[r]^i \ar[d]_\psi^\simeq & \cublat \ar[d]_\simeq^\phi \\
        T(X) \ar@{^(->}[r]^{i'} & \cublat,
        }
    \]
    where $\psi\in O_{\mathrm{Hodge}}(T(X))$, if and only if   
    there is an isometry $\psi'\colon N'\simeq N'$ fitting into a commutative diagram 
    \begin{equation*}
            \xymatrix{
                G_i \ar[r]^{\gamma_i} \ar[d]_{\overline{\psi}}& A_{N'} \ar[d]^{\overline{\psi'}} \\
                G_{i'}\ar[r]^{\gamma_{i'}}& A_{N'}.
            }
    \end{equation*}
    This exactly means that $\gamma_i$ and $\gamma_{i'}$ define the same element of $O(N')\backslash O(A_{T(X)}) / O_{\mathrm{Hodge}}(T(X)),$ hence the map $\Gamma_{N'}$ is well-defined and injective.

    For surjectivity, note that any isometry $\gamma\colon A_{T(X)}\simeq A_{T(X)}$ induces the pair $(A_{T(X)},\gamma)$. From this pair, we obtain a primitive embedding $T(X)\hookrightarrow \cublat$  with orthogonal complement isometric to $N'$ as in Remark \ref{rem: pair inducing a primitive embedding}. This shows that $\Gamma_{N'}$ is surjective. 
    
    The result now follows from Corollary \ref{cor: counting primitive embeddings}.
\end{proof}

Due to the symmetry of the situation, one might expect an similar counting formula for Fourier--Mukai partners of cubic fourfolds to hold when we assume $\disc(A(X)_\mathrm{prim})$, rather than $\disc(T(X))$, is not divisible by $3$. 
The fundamental difference between these two assumptions is the following. If $X$ is a cubic fourfold such that $\disc(T(X))$ is not divisible by $3$, then $\disc(T(Y))$ is also not divisible by $3$ for any Fourier--Mukai partner $Y$ of $X$. 
On the other hand, if we assume $\disc(A(X)_\mathrm{prim})$ to be indivisible by $3$, then $X$ may have several Fourier--Mukai partners $Y$ for which $\disc(A(Y)_\mathrm{prim})$ is a multiple of $3$. Nevertheless, we have the following Counting Formula for Fourier--Mukai partners of $X$ whose primitive algebraic lattice has discriminant not divisible by $3$.

\begin{proposition} \label{prop: counting formula div 3}
    Let $X$ be a cubic fourfold which satisfies Derived Torelli, and write $N\coloneqq A(X)_\mathrm{prim}$. Suppose that $\disc(N)$ is not divisible by $3$. Then for any $N'\in \mathcal{H}(N)$, we have 
    \begin{equation} \label{eq: counting formula div 3}
    \#\FM(X,N') = |O(N')\backslash O(A_N)/O_{\mathrm{Hodge}}(T(X))|.
    \end{equation}
    Here, the action of $O_\mathrm{Hodge}(T(X))$ on $O(A_{N})$ is the one defined by Remark \ref{rem: isometry actions}.
\end{proposition}

\begin{proof}
    We fix an isometry $A_{T(X)} \simeq A_{N'}\oplus C_3$, where $C_3$ is one of the two non-degenerate finite quadratic modules of order $3$. A primitive embedding $i\colon T(X)\hookrightarrow \cublat$ induces a pair $(G_i,\gamma_i)$, where $G_i\subset A_{T(X)}$ and $\gamma_i\colon G_i\hookrightarrow A_{N'}$. 
    The group $G_i$ is equal to $A_{N'}$ as a subgroup of $A_{T(X)}$, since $A_{N'}\subset A_{T(X)}$ is the unique maximal subgroup whose order is not divisible by $3$. Therefore, $\gamma_i$ can be viewed as an element of $O(A_{N'}).$

    Similarly to the proof of Theorem \ref{thm: counting formula}, we obtain a map
    \[
    \isomorphism{\Gamma_{N'}}{\Emb_{N'}(T(X),\cublat)/O_{\mathrm{Hodge}}(T(X))}{O(N')\backslash O(A_{T(X)}) / O_{\mathrm{Hodge}}(T(X))}{i}{\gamma_i.}
    \]
    which we claim to be a bijection. The proof the bijectivity of $\Gamma_{N'}$ is completely analogous to the proof of Theorem \ref{thm: counting formula}, so we omit it here.
\end{proof}

\begin{remark}
    In the setting of Proposition \ref{prop: counting formula div 3}, it is possible that $X$ has a Fourier--Mukai partner $Y$ with $N ' = A(Y)_\mathrm{prim}\notin \mathcal{H}(N)$. This is equivalent to the discriminant of $A(Y)_\mathrm{prim}$ being divisible by $3$. 
    In this case, the gluing group of the embedding $N'\hookrightarrow H^4(Y,\Z)_\mathrm{prim}$ is isomorphic to $A_{T(Y)}$, but it is no longer necessarily the case that $A_{N'}$ can be written as $A_{T(Y)} \oplus C_3$ for some finite quadratic form $C_3$ of order $3$. 
    It is still possible to compute $\FM(X,N')$ via Proposition \ref{prop: correspondence between FM partners and primitive embeddings}, c.f. \cite{BBM25}, but a simple formula such as \eqref{eq: counting formula div 3} seems to be unattainable in this case.
    See Remark \ref{rem: the other FM partners for L29} for examples of this phenomenon.
    \end{remark}

\section{Examples of $\LL$-equivalent cubic fourfolds}
In this section, we construct new examples of $\LL$-equivalent cubic fourfolds, using the counting formula for Fourier--Mukai partners of Section \ref{sec: counting formula}. These examples are based on the following result.

\begin{proposition}\cite[\S 5]{IMOU20} \cite[\S 2.6.1]{KS18} \label{prop: motives of certain cubic fourfolds}
    Let $X$ be a cubic fourfold with one of the following properties:
    \begin{itemize}
        \item[A)] $X\in \mathcal{C}_{14}$, i.e. $X$ is a Pfaffian cubic fourfold
        \item[B)] $X$ contains a plane $P$ and another cycle $T\in \operatorname{CH}^2(X)$ such that the intersection number $(h^2-P)\cdot T$ is odd
    \end{itemize}
    Then there exists a K3 surface $S$ such that there exists a Hodge isometry $T(S)(-1)\simeq T(X)$ and such that the class of $X$ in the Grothendieck ring of varieties is given by
    \[
    [X] = 1 + \LL^2 + \LL^4 + [S]\LL.
    \]
\end{proposition}

It may be possible to extend \ref{prop: motives of certain cubic fourfolds} to other classes of cubic fourfolds, for instance the case of $\mathcal{C}_{42}$ by using the techniques of \cite{RS23}. We plan to return to this question in the future. 

In light of Proposition \ref{prop: motives of certain cubic fourfolds}, Pfaffian cubic fourfolds seem like a good place to start looking for $\LL$-equivalent cubic fourfolds. However, very general Pfaffian cubic fourfolds have no Fourier--Mukai partners. This result follows immediately from \cite[Proposition 2.6]{FL23}, and also from \cite[Theorem 1.1]{PertFMPartners21} combined with \cite[Proposition 1.10]{Ogu01}. Here, we give a proof using Theorem \ref{thm: counting formula}.

\begin{proposition}
    Let $X$ be a cubic fourfold that is a very general member of $\mathcal{C}_{14}$, that is, such that $\rk A(X) =2$. Then $\FM(X) =\{X\}$.
\end{proposition}
\begin{proof}
    Since $A(X) \simeq K_{14}$ has rank $2$, it follows that $\rk A(X)_\text{prim} = 1$. Therefore, $N\coloneqq A(X)_\text{prim}$ is clearly unique in its genus. 
    Moreover, we have $\disc(T(X)) = 14$, and $A_{T(X)}$ is cyclic, since it is naturally a subgroup of the cyclic group $A_N$ by Remark \ref{rem: isometry actions}. This implies that $A_{T(X)}$ has precisely $2$ isometries, namely $\pm \id_{A_{T(X)}}$. 
    Indeed, if $x\in A_{T(X)}$ is a generator, then for any $\sigma\in O(A_{T(X)})$, we must have $\sigma(x) = ax$ for some $a\in (\Z/14\Z)^\times$. However, since $x^2 = \sigma(x)^2 = a^2x^2 \pmod {2\Z}$, 
    it follows that we must have $a^2 \equiv 1 \pmod{14}$. Since $\pm 1$ are the only elements of $(\Z/14\Z)^\times$ whose square equals $1$, we find $O(A_{T(X)}) = \{\pm \id_{A_{T(X)}}\}$.
    Since both of these isometries are in the image of the group homomorphism $O_{\mathrm{Hodge}}(T(X)) \to O(A_{T(X)})$, it follows from Theorem \ref{thm: counting formula} that $\#\FM(X) = 1$.     
\end{proof}

Similarly, one might expect examples of $\LL$-equivalent cubic fourfolds when studying cubic fourfolds containing two disjoint planes. If $X$ is a cubic fourfold containing two planes, which is very general with this property, then the algebraic part $A(X)$ of $H^4(X,\Z)$ is isometric to the lattice
\begin{equation} \label{eq: bilinear form with two planes}
     \left(
    \begin{matrix}
        3 & 1 & 1 \\
        1 & 3 & 0 \\
        1 & 0 & 3
    \end{matrix}
    \right)
\end{equation}

\begin{remark} 
    Some of the lattice theoretical computations involving the next examples were done with OSCAR \cite{OSCAR, OSCAR-book}.
    We briefly recall here the behaviour of the key functions we used.
    \begin{itemize}
    \item \texttt{primitive\_embeddings(L, N)} (developed by Stevell Muller): computes primitive embeddings of the lattice \texttt{N} into \texttt{L} (optionally with classification data such as orthogonal complements).
    \item \texttt{vectors\_of\_square\_and\_divisibility(L, m, d)} (developed by Stevell Muller): returns vectors $v\in L$ with $v^2=m$ and $\dv(v)=d$.
    \item \texttt{genus\_representatives(L)} (developed by Simon Brandhorst and Stevell Muller): returns a set of lattices representing the genus of \texttt{L}.
    \end{itemize}
The complete OSCAR code for checking the data of this paper can be found in the file ancillary to this paper.
\end{remark}

\begin{proposition} \label{prop: two planes, no fm partners}
    Let $X$ be a cubic fourfold whose algebraic lattice $A(X)$ is isometric to \eqref{eq: bilinear form with two planes}. Then $\FM(X) = \left\{X\right\}$.
\end{proposition}
\begin{proof}
    Since $\disc(T(X)) = \disc(A(X))= 21$ is divisible by $3$, Theorem \ref{thm: counting formula} is not applicable in this situation. Therefore, we use the counting algorithm of \cite{BBM25}.

    First, we determine that there is a unique lattice $N$ that arises as an orthogonal complement of an embedding $T(X)\hookrightarrow \cublat$. Explicitly, this lattice is given by 
    \[
    N \coloneqq A(X)_\text{prim}\simeq 
    \left( 
    \begin{matrix}
        12 & -3 \\ -3 & 6
    \end{matrix}
    \right).
    \]
    Since $\disc(N) = 63$, it follows from Lemma \ref{lem: gluing group orders} that the gluing group $G = A_{T(X)}$, hence $A_{N} \simeq A_{T(X)}(-1)\oplus C_3$, where $C_3$ is one of the two finite quadratic modules of order $3$. 
    Following the same reasoning as in the proof of Theorem \ref{thm: counting formula}, we obtain 
    \[
    \#\FM(X) = |O(N)\backslash O(A_{T(X)}) / O_{\mathrm{Hodge}}(T(X))|, 
    \]
    where the action of $O(N)$ on $O(A_{T(X)})$ is the one defined in Remark \ref{rem: isometry actions}.
    The natural map $O(N) \to O(A_N)$ is surjective. Since $O(A_N) \simeq O(A_{T(X)}(-1))\oplus O(C_3)$, it follows that $O(N)$ acts transitively on $O(A_{T(X)})$. Therefore, we have $\#\FM(X) = 1$, as required. 
\end{proof}

We now turn our attention to the lattices of the form

\[
    L_n \coloneqq 
    \left(
        \begin{matrix}
            3 & 1 & 1 \\
            1 & 3 & 0 \\
            1 & 0 & n
        \end{matrix}
    \right),
\]
where $n\geq 1$. 
If $X$ is a cubic fourfold with $A(X)\simeq L_n$, then $X$ satisfies property B of Proposition \ref{prop: motives of certain cubic fourfolds}, so we are able to compute the class of $X$ in $\gravc$. 

Importantly, if $(S_1$, $S_2$, $S_3)$ is an ordered $\Z$-basis for $A(X)$ whose associated Gram matrix is $L_n$, then $S_1$ is the square of a hyperplane class if and only if $n$ is odd. On the other hand, if $n$ is even, then $S_2$ is the square of a hyperplane class.

By Proposition \ref{prop: two planes, no fm partners}, $X$ has no Fourier--Mukai partners when $n = 3$. Also, since the discriminant of $L_{n}$ is $\disc(L_n) = 8n-3$, we may use Theorem \ref{thm: counting formula} to compute the number of Fourier--Mukai partners of $X$ whenever $n$ is not divisible by $3$. 

\begin{proposition}\label{prop: L10 examples}
    Let $X$ be a cubic fourfold such that $A(X) \simeq L_{10}$.
    Assume moreover that $O_{\mathrm{Hodge}}(T(X)) \simeq \Z/2\Z$.
    Then $\#\FM(X) = 2$.
    Moreover, the unique non-trivial Fourier--Mukai partner $Y$ of $X$ has $A(Y)_\text{prim}\simeq A(X)_\text{prim}$. In addition, $X$ and $Y$ are trivially $\LL$-equivalent.
\end{proposition}
\begin{proof}
    Since $\disc(T(X)) = \disc(L_{10}) = 77$ is not divisible by $3$, we may compute $\#\FM(X)$ using Theorem \ref{thm: counting formula}. 
    A direct computation reveals 
    \[
        N\coloneqq A(X)_\text{prim} \simeq 
        \left(
            \begin{matrix}
                24 & -3 \\ -3 & 10
            \end{matrix}
        \right).
    \]
    
    This lattice is not unique in its genus. Indeed, we have 
    $
    \mathcal{G}(N) = \left\{N,N'\right\}
    $
    where 
    \[
        N' \coloneqq 
        \left(
            \begin{matrix}
                6 & -3 \\ -3 & 40
            \end{matrix}
        \right).
    \]
    Note that $N'$ contains a vector $v$ with $v^2 = 6$ and $\dv(v) = 3$. Therefore, $\mathcal{H}(N) = \left\{N\right\}$.
    Since $A_{T(X)}$ has order $77$, it follows that $O(A_{T(X)})\simeq \Z/2\Z \oplus \Z/2\Z$. 
    However, we have $O(N) \simeq \Z/2\Z$, hence 
    \[
    |O(N)\backslash O(A_{T(X)}) / O_{\mathrm{Hodge}}(T(X))| = 2. 
    \]
    That is, by Theorem \ref{thm: counting formula}, $X$ has a unique Fourier--Mukai partner $Y$, and we have $A(Y)_{\text{prim}} \simeq A(X)_{\text{prim}} \simeq N$. 
    
    It remains to be shown that $X$ and $Y$ are $\LL$-equivalent. To prove this, we argue that $A(X)\simeq A(Y)$. Indeed, since $T(X)\simeq T(Y)$, and $A(X) = T(X)^\perp \subset H^4(X,\Z)$, it follows that $A(X)$ and $A(Y)$ are in the same genus. 
    However, the genus of $L_{10}$ consists of 5 lattices.  Since $N \oplus \langle 3 \rangle \simeq A(Y)_\text{prim}\oplus \langle 3 \rangle \subset A(Y)$, contains both a vector of square $10$ and a vector of square $3$. Of the 5 lattices in $\mathcal{G}(L_{10})$, only $L_{10}$ and the lattice
    \[
    L_{10}' \coloneqq 
    \left(
        \begin{matrix}
            3 & -1 & 0 \\
            -1 & 4 & 0 \\
            0 & 0 & 7
        \end{matrix}
    \right)
    \]
    contain both a vector of square $10$ and a vector of square $3$.
    However, $L_{10}'$ contains exactly 1 vector of square $3$, and its orthogonal complement is 
     \[
        \left(
            \begin{matrix}
                7 & 0 \\ 0 & 33
            \end{matrix}
        \right),
    \]
    which is not an even lattice. However, since $A(Y)_\text{prim}\subset A(Y)$ is an even lattice which is the orthogonal complement of a vector of square $3$, $A(Y)$ cannot be isometric to $L_{10}'$ and thus it must be isometric to $L_{10}$. In particular, both $X$ and $Y$ satisfy property B of Proposition \ref{prop: motives of certain cubic fourfolds}, hence there exist K3 surfaces $S, S'$ such that 
    \[
    [X] = 1+ \LL^2 + \LL^4 + [S]\LL, \qquad [Y] = 1+ \LL^2 + \LL^4 + [S']\LL. 
    \]
    Moreover, we also obtain a Hodge isometry $T(S)(-1)\simeq T(X)\simeq T(Y) \simeq T(S')(-1)$ from Proposition \ref{prop: motives of certain cubic fourfolds}. In particular $S$ and $S'$ are Fourier--Mukai partners. 
    However, we compute that 
    \[
    \NS(S) \simeq 
    \left(
        \begin{matrix}
            2 & 7 \\
            7 & -14
        \end{matrix}
    \right),
    \]
    and one can compute using \cite[Theorem 2.3]{HLOY02} that $S$ has no non-trivial Fourier--Mukai partner, hence $S\simeq S'$.
    In conclusion, we have
    \[
    [X] = [Y] \in \gravc,
    \]
    so that $X$ and $Y$ are trivially $\LL$-equivalent.
\end{proof}

\begin{remark}
    The positive integer $n=10$ is the smallest for which the lattice $L_n$ gives an example. The next ones are $n=11, 17, 26, 28$. A complete classification is not known, but may be possible with the same techniques, as well as experimenting with other lattices. It would also be interesting to give a geometric description of such cubic fourfolds.
\end{remark}

We now turn our attention to cubic fourfolds in $\mathcal{C}_{14}$. For $a, b \in \Z$, consider the lattices
$$
L_{a,b}\coloneqq
\begin{pmatrix}
3 & 4 & 0\\
4 & 10 & a\\
0 & a & 2b
\end{pmatrix}
$$
Notice first that this family describes all cubic fourfolds in $\mathcal{C}_{14}$ whose primitive algebraic lattice has rank two, or equivalently whose full algebraic lattice has rank three. Indeed, let $X\in\mathcal{C}_{14}$ with $\operatorname{rk}A(X)=3$. We may choose classes $h^2,\tau\in A(X)$ spanning a primitive copy of $K_{14}$.
Since $(h^2,h^2)=3$ and $(h^2,\tau)=4$ are coprime, a third basis vector may be modified by an integral linear combination of $h^2$ and $\tau$ so as to become orthogonal to $h^2$. Its square is even, because $(h^2)^\perp$ is an even lattice. With respect to the resulting basis, the Gram matrix of $A(X)$ therefore has the form $L_{a,b}$ for suitable integers $a,b$.

Conversely, if $L_{a,b}$ is realised as the algebraic lattice of a cubic fourfold, with the first basis vector corresponding to $h^2$, then its first two basis vectors span a primitive copy of $K_{14}$. Hence the corresponding cubic fourfold belongs to $\mathcal{C}_{14}$.

The lattice $L_{a,b}$ is positive definite precisely when $\disc(L_{a,b})=28b-3a^2>0$. Moreover, its primitive part is generated by $3\tau-4h^2$ and the third basis vector, and therefore
$$
A(X)_{\mathrm{prim}}\simeq
N_{a,b}\coloneqq
\begin{pmatrix}
42 & 3a\\
3a & 2b
\end{pmatrix}.
$$
In particular, $\disc(N_{a,b})=3(28b-3a^2)$. Thus the discriminant of the primitive algebraic lattice is always divisible by $3$, so Proposition \ref{prop: counting formula div 3} cannot be applied to this family. 

\begin{proposition}\label{prop: L29 examples}
Let $X$ be a cubic fourfold whose algebraic lattice is isometric to $L_{2,7}$
with the first basis vector corresponding to $h^2$, and assume that
$$
O_{\mathrm{Hodge}}(T(X))=\{\pm\id\}.
$$
Then $X$ has a unique non-isomorphic Fourier--Mukai partner $Y$ such that
\[
A(Y)_{\mathrm{prim}}\simeq A(X)_{\mathrm{prim}}.
\]
Moreover, $A(Y)\simeq L_{2,7}$ and $X,Y\in\mathcal C_{14}\setminus\mathcal C_8$.
In particular, they are Pfaffian and trivially $\LL$-equivalent.
\end{proposition}

\begin{proof}
We have $\disc(L_{2,7})=184$ and $\disc(N_{2,7})=552$.
In particular, $L_{2,7}$ is positive definite. Moreover, for $v=(x,y)\in N_{2,7}$,
\[
v^2
=
42x^2+12xy+14y^2
=
14\left(y+\frac{3x}{7}\right)^2+\frac{276}{7}x^2.
\]
It follows that every non-zero vector of $N_{2,7}$ has square at least $14$. Hence $N_{2,7}$ contains neither vectors of square $2$ nor vectors of square $6$.

An OSCAR computation shows that $N_{2,7}(-1)$ admits primitive embeddings into the primitive cubic lattice giving two distinct transcendental components, which we denote by $T_1$ and $T_2$. The component corresponding to $T_1$ gives four primitive-extension classes, but none of them contain a primitive copy of $K_{14}$, and none have full algebraic lattice isometric to $L_{2,7}$. Thus this component does not contain a cubic fourfold as in the statement.

The component corresponding to $T_2$ contains a realisation of the prescribed marked lattice $L_{2,7}$. In this component, the primitive-extension computation gives precisely two Fourier--Mukai classes with primitive algebraic lattice isometric to $N_{2,7}$. Both extensions contain a primitive copy of $K_{14}$ and have full algebraic lattice isometric to $L_{2,7}$. Hence both classes are represented by cubic fourfolds in $\mathcal C_{14}$.

By Derived Torelli, these two classes are represented by $X$ and by a unique non-isomorphic Fourier--Mukai partner $Y$.
Moreover, the lattice computation shows that neither of the two extensions in the $T_2$-component belongs to $\mathcal C_8$. Thus $X,Y\in\mathcal C_{14}\setminus\mathcal C_8$.

Since both belong to $\mathcal C_{14}$, they are Pfaffian. In particular, both satisfy Property A of Proposition \ref{prop: motives of certain cubic fourfolds}.

We consequently find K3 surfaces $S$ and $S'$ such that
$$
[X]=1+\LL^2+\LL^4+[S]\LL,
\qquad
[Y]=1+\LL^2+\LL^4+[S']\LL,
$$
and a Hodge isometry
$$
T(S)(-1)\simeq T(X)\simeq T(Y)\simeq T(S')(-1).
$$
In particular, $S$ and $S'$ are Fourier--Mukai partners. The computation of the orthogonal complement of $T(S)\simeq T(X)(-1)$ in the K3 lattice gives
$$\NS(S)\simeq
\begin{pmatrix}
-2 & 12\\
12 & 20
\end{pmatrix}.$$
After a unimodular change of basis, this can equivalently be written as
$$\NS(S)\simeq
\begin{pmatrix}
14 & -10\\
-10 & -6
\end{pmatrix},$$
where the first basis vector is a primitive class of square $14$. Applying \cite[Theorem 2.3]{HLOY02} to this lattice we get that $S$ has no non-trivial Fourier--Mukai partners. Since $S'$ is a Fourier--Mukai partner of $S$, it follows that $S\simeq S'$, and therefore
$$[X]=[Y]\in\gravc.$$
Thus $X$ and $Y$ are trivially $\LL$-equivalent.
\end{proof}

\begin{remark}\label{rem: the other FM partners for L29}
The same computations give further examples of this type, for instance
\[
(a,b)=(1,8),\ (1,9),\ (1,11),\ (2,5),\ (2,10).
\]
For each of these pairs, in the appropriate transcendental component there are exactly two cubic Fourier--Mukai classes with the prescribed primitive algebraic lattice; both belong to $\mathcal{C}_{14}$, and the associated K3 surface has no non-trivial Fourier--Mukai partners. Consequently, the same conclusion holds whenever the two cubic fourfolds satisfy Property A of Proposition \ref{prop: motives of certain cubic fourfolds}.
\end{remark}

\section{Implications of $\LL$-equivalence}

The following proposition is the analogue for cubic fourfolds of \cite[Proposition 3.3]{Efi18}. We remark that for cubic fourfolds which satisfy Derived Torelli, this follows immediately from \cite[Theorem 1.1]{Huy16}.

\begin{proposition} \label{prop:Finitely many 4fold with iso Hodge}
    Let $X$ be a cubic fourfold. Then there exist finitely many isomorphism classes of cubic fourfolds $Y$ such that there is an isomorphism (but not necessarily an isometry) of Hodge structures $f\colon  H^4(X,\bbZ) \to H^4(Y,\bbZ)$.
\end{proposition}
\begin{proof}
Passing to Hodge structures of weight $2$ we set $H_X := H^4(X,\bbZ)(1)$ and $H_Y:=H^4(Y,\bbZ)(1)$. We have induced isomorphisms of algebraic and
transcendental lattices (see Definition~\ref{defn:transcendental lattice
of cubic 4fold}):
$$A(X) = \NS(H_X) \cong \NS(H_Y) = A(Y), \qquad
T(X)(1) = T(H_X) \cong T(H_Y) = T(Y)(1).$$
In particular, by Lemma \ref{lem: L equivalent cubic fourfolds have equal discriminants} we have also an equality $D:=\disc(T(X)) = \disc(T(Y))$.

Denote by $S_D$ the set of all symmetric pairings $T(H_X)\otimes T(H_X) \to \bbZ(-2)$ which are morphisms in $\operatorname{HS}_{\bbZ,4}$ and have discriminant $D$. The isomorphism $f(1)$ pulls back the intersection form on $H^4(Y,\bbZ)$, and restricting to the transcendental part gives an element of $S_D$. 

Consider the endomorphism ring $R := \End_{\operatorname{HS}_{\bbZ,2}}(T(H_X))$, and 
the group $G := \Aut_{\operatorname{HS}_{\bbZ,2}}(T(H_X))$ with its natural action on $S_D$. Up to Hodge automorphisms of $T(H_X)$, the possible transcendental intersection forms
coming from such $Y$ are parametrized by the $G$-orbits in $S_D$.

Now we need to show that the number of $G$-orbits in $S_D$ are finite. We follow the same proof as in \cite[Proposition~3.3]{Efi18}. The only change is that we use $T(H_X)$ instead of the transcendental lattice of a K3 surface. 
For carrying on the proof we need to check that $T(H_X)$ satisfies the hypotheses of Zarhin’s theorem \cite{Zarhin1983} (namely that $T(H_X)$ is a polarized weight $2$ Hodge structure of K3 type). As a consequence, the $\bbQ$-algebra $R \otimes \bbQ$ is a totally real number field, or an imaginary quadratic extension of a totally real number field. Denote by $\beta$ an element of $S_D$ and consider the associated Hodge endomorphism $a(\beta) \in F$. Efimov's argument then applies verbatim and gives the finiteness of $S_D/G$.

Finally, for each of these finitely many orbits we can apply Lemma \ref{lemma:Finitely many 4fold with isometric Hodge} and get finitely many isomorphism classes of cubic fourfold $Y$, concluding the proof.
\end{proof}

\begin{lemma} \label{lemma:Finitely many 4fold with isometric Hodge}
    Let $X$ be a cubic fourfold. Then there are exist finitely many isomorphism classes of cubic fourfolds $Y$ such that there is a Hodge isometry between $H^4(X,\bbZ)$ and $H^4(Y,\bbZ)$.
\end{lemma}
\begin{proof}
    Assume first we have a Hodge isometry $\phi\colon H^4(X,\bbZ) \to H^4(Y,\bbZ)$, and consider the class $v \coloneqq \phi^{-1}(h_Y^2)$. Since $\phi$ is a Hodge morphism and $h_Y^2 \in A(Y)$, we have $v \in A(X)$, and since $\phi$ is an isometry $(v,v) = (h_Y^2,h_Y^2) = 3$. The lattice $A(X)$ is positive definite. In a positive definite integral lattice, there are only finitely many vectors of a given square, see \cite[Section 2.1]{Hanke2013} for a reference. Hence the set $S_3 := \{ v\in A(X) : (v,v)=3\}$ is finite. Fix $v \in S_3$ and assume $Y_1$ and $Y_2$ are cubic fourfolds with Hodge isometries $\phi_i\colon H^4(X,\bbZ) \to H^4(Y_i,\bbZ)$ such that $\phi_i^{-1}(h_{Y_i}^2)=v$. The composition $\phi_2 \circ \phi_1^{-1}$ is a Hodge isometry between $H^4(Y_1,\bbZ)$ and $H^4(Y_2,\bbZ)$. Moreover, this composition sends $h_1^2$ to $h_2^2$ and so gives an identification  $H^4(Y_1,\bbZ)_{\prim} \cong H^4(Y_2,\bbZ)_{\prim}$. By Theorem \ref{thm: torelli for cubic fourfolds} we have $Y_1 \cong Y_2$.

Since $S_3$ is finite, there are only finitely many isomorphism
classes of cubic fourfolds $Y$ admitting a Hodge isometry
$H^4(X,\bbZ)\cong H^4(Y,\bbZ)$.
\end{proof}

The following theorem is the analogue for cubic fourfolds of \cite[Theorem 3.4]{Efi18}. Again, for cubic fourfolds which satisfy Derived Torelli, Theorem \ref{thm: L implies FM} implies that $\LL$-equivalent cubic fourfolds are Fourier--Mukai partners, and the results again follows from \cite[Theorem 1.1]{Huy16}.

\begin{theorem} \label{thm: L equivalences classes are finite}
    Let $X$ be a cubic fourfold. Then there exist finitely many cubic fourfolds $Y$ that are $\LL$-equivalent to $X$.
\end{theorem}
\begin{proof}
Let $Y$ be a cubic fourfold such that $X \sim_\LL Y$, and consider the integral Hodge realisation $\operatorname{Hdg}_\bbZ$ constructed in \cite[Section 3]{Efi18}. From $[X]=[Y]$ in $\gravc[\LL^{-1}]$ we deduce $[H^4(X,\bbZ)] = [H^4(Y,\bbZ)]$ in $K_0(\operatorname{HS}^4_\bbZ)$.

By \cite[Theorem 2.7]{Efi18}, applied to the additive category $\operatorname{HS}^4_\bbZ$, the class $[H^4(X,\bbZ)]$ in $K_0(\operatorname{HS}^4_\bbZ)$ contains only finitely many isomorphism classes of objects. In particular there are only finitely many possibilities for $H^4(Y,\bbZ)$ as an integral Hodge structure when $Y$ ranges over cubic fourfolds with $X \sim_\LL Y$. The thesis then follows from Proposition \ref{prop:Finitely many 4fold with iso Hodge}.
\end{proof}

The following result appears as Theorem 1.4 in \cite{BB2025}.
\begin{theorem} \label{thm: L implies FM}
    Let $X$ and $Y$ be cubic fourfolds which satisfy Derived Torelli (for example, if $X$ is a cubic fourfold satisfying one of the properties of Proposition \ref{prop: derived torelli criteria}) Moreover, assume $\rk(T(X)) \neq 4$ and $\operatorname{End}_{\mathrm{Hodge}}(T(X))\simeq \Z$. Then, if $X$ and $Y$ are $\LL$-equivalent, they are Fourier--Mukai partners.
\end{theorem}
\begin{proof}
    By Theorem \ref{thm: L equivalence implies Hodge isometry}, there is a Hodge isometry $T(X)\cong T(Y)$. Then, if $X$ and $Y$ satisfy Derived Torelli, then they are Fourier--Mukai partners.
\end{proof}

The following result appears as Theorem 1.2 in \cite{BB2025}.

\begin{proposition} \label{prop: very general cubic fourfolds have no L partners}
    Let $X$ be a cubic fourfold not contained in any of the Hassett divisors. Assume moreover that $\operatorname{End}_{\mathrm{Hodge}}(T(X))\simeq \Z$. Then, for any cubic fourfold $Y$, we have 
    \[
    X\sim_\LL Y \iff X\simeq Y
    \]
\end{proposition}
\begin{proof}
    Note that $\rk(T(X)) = 22 \neq 4$. Therefore, we may apply Theorem \ref{thm: L equivalence implies Hodge isometry} to obtain a Hodge isometry $T(X)\simeq T(Y)$. 
    However, since $X$ and $Y$ are not contained in any Hassett divisor, we have \[H^4(X,\Z)_{\operatorname{prim}} = T(X) \simeq T(Y) = H^4(Y,\Z)_{\operatorname{prim}},\] hence $X\simeq Y$ by the Torelli Theorem \ref{thm: torelli for cubic fourfolds}.
\end{proof}

The following result appears as Proposition 4.4 in \cite{BB2025}.
\begin{proposition} \label{prop: LE implies LF}
    Let $X$ and $Y$ be cubic fourfolds. Suppose $X$ and $Y$ are $\LL$-equivalent. Then their Fano varieties of lines are also $\LL$-equivalent.
\end{proposition}
\begin{proof}
    Since $X$ and $Y$ are assumed to be $\LL$-equivalent, $X^{[2]}$ and $Y^{[2]}$ are $\LL$-equivalent by \cite[Lemma 2.1]{Oka21}.
    Now the result follows directly from the \textit{beautiful formula} of Galkin--Shinder \cite{GS14}: for any cubic fourfold $X$, we have
    \[\LL^2[F(X)] = [X^{[2]}] - [\PP^4][X].\]
\end{proof}

\begin{theorem} \label{thm: LF implies FM and DF}
    Let $X$ and $Y$ be cubic fourfolds. Assume that $X$ satisfies Derived Torelli, that $\rk(T(X))\neq 4$, and that $\End_{\mathrm{Hodge}}(T(X)) \simeq \Z$. Then if $F(X)\sim_\LL F(Y)$, there is an equivalence $\mathcal{A}_X\simeq \mathcal{A}_Y$, hence also an equivalence $\Db(F(X))\simeq \Db(F(Y))$. 
\end{theorem}
\begin{proof}
    By \cite[Theorem 3.5]{Mei25} and Lemma \ref{lem: coarseness of a fano variety}, the gluing group $$G= \frac{H^2(F(X),\Z)}{T(F(X))\oplus \NS(F(X))}$$ has order $|G| = \disc(T(F(X)))$. Therefore, if $F(X)$ is $\LL$-equivalent to $F(Y)$, it follows from Lemma \ref{lem: gluing group homomorphism} that $\disc(T(F(X))) = \disc(T(F(Y)))$. Since $T(F(X))$ is Hodge anti-isometric to $T(X)$, we find that $[T(X)(1)] = [T(Y)(1)]$. By Theorem \ref{thm: L equivalence implies Hodge isometry}, there is a Hodge isometry $T(X)\simeq T(Y)$, hence an equivalence $\mathcal{A}_X\simeq \mathcal{A}_Y$ by Definition \ref{def: derived torelli assumption}. The fact that there is an equivalence $\Db(F(X)) \simeq \Db(F(Y))$ follows from the main result of \cite{KS25}.
\end{proof}

The proof of the following result was communicated to us by Evgeny Shinder.
\begin{proposition} \label{prop: birational implies same motive}
Let $X$ and $Y$ be projective hyperk\"ahler manifolds. Suppose $X$ and $Y$ are birational. Then we have $[X] = [Y] \in \gravc$. 
\end{proposition}
\begin{proof}
    By \cite[Theorem 4.6]{Huy1999} there exist two smooth projective families $\calX$ and $\calY$ over a smooth one dimensional base $S$ with special fibres $\calX_0 \cong X$ and $\calY_0 \cong Y$, such that $\calX$ and $\calY$ are isomorphic over $S \smallsetminus \{0\}$. 

    Now consider the generic fibres over $S \smallsetminus \{0\}$. As in \cite{NS19} let $R$ be the completion of the local ring of $S$ at $0$ and let $K = \operatorname{Frac}(R)$, so that $\operatorname{Spec} K$ is the generic point of $\operatorname{Spec}(R)$. By taking base changes we get the restrictions $\calX_R$ and $\calY_R$ of $\calX$ and $\calY$ to a formal neighbourhood of $0$, and the corresponding generic fibres $\calX_K$ and $\calY_K$. 
    
     Recall that we have a canonical identification of special fibres
$(\calX_R)_k=\calX_R\times_{\Spec R}\Spec k \cong \calX\times_S\Spec k \cong  \calX_0$,
and similarly $(\calY_R)_k\cong \calY_0$.
    
    The isomorphism of the families over $S \smallsetminus \{0\}$ induces an isomorphism over $K$ of $\calX_K$ and $\calY_K$, which induces $[\calX_K]=[\calY_K]$ in $K_0(\operatorname{Var}_K)$. 
    
    Now apply the motivic volume morphism
$\operatorname{Vol}_K: K_0(\operatorname{Var}_K)\to K_0^{\hat\mu}(\operatorname{Var}_\bbC)$ of
\cite[Theorem 3.1.1]{NS19}. From $[\calX_K]=[\calY_K]$ we get $\operatorname{Vol}_K([\calX_K])=\operatorname{Vol}_K([\calY_K])$.
From the fact that $\calX_R$ ($\calY_R$) is a smooth proper $R$ model of $\calX_K$ ($\calY_K$), so $\operatorname{Vol}_K([\calX_K])=[\calX_0]=[X]$ and $\operatorname{Vol}_K([\calY_K])=[\calY_0]=[Y]$. This concludes the proof.
\end{proof}

We obtain the following as an immediate consequence:
\begin{corollary} \label{cor: BF implies LF}
    Let $X$ and $Y$ be cubic fourfolds. If $F(X)$ is birational to $F(Y)$, then $F(X)$ and $F(Y)$ are $L$-equivalent. In other words, we have 
    $\mathbf{BF} \implies\mathbf{LF}$.
\end{corollary}
\printbibliography
\end{document}